\UseRawInputEncoding
\documentclass[12pt]{article}

\usepackage{dsfont}
\usepackage{amsfonts,amsmath,amsthm}
\usepackage{algorithm, algorithmicx}
\usepackage{color}
\usepackage{graphicx}
\usepackage{stmaryrd}        
\usepackage{upgreek}
 
\usepackage[table,xcdraw]{xcolor}
\usepackage{makecell}
\usepackage{subcaption}
\usepackage{xcolor}    
\usepackage{colortbl}
\usepackage{tabularray}
\usepackage{mathtools}

\usepackage{multicol}
\usepackage{multirow}
\usepackage{url}
\usepackage[paper=a4paper,dvips,top=2cm,left=2cm,right=2cm,
foot=1cm,bottom=4cm]{geometry}

\usepackage{booktabs}

\newcommand{\vx}{\mathbf{x}}
\newcommand{\vv}{\mathbf{v}}
\newcommand{\vu}{\mathbf{u}}
\newcommand{\vw}{\mathbf{w}}
\newcommand{\vy}{\mathbf{y}}

\newcommand{\veta}{\boldsymbol{\upeta}}
\newcommand{\ii}{\mathbf{i}}
	\newcommand{\jj}{\mathbf{j}}
	\newcommand{\kk}{\mathbf{k}}
\newcommand{\diag}{\textup{diag}}
    
\usepackage{bm}
\usepackage{algpseudocode}
\usepackage{hyperref}
\usepackage{enumitem}
\newtheorem{theorem}{Theorem}[section]
\newtheorem{lemma}[theorem]{Lemma}

\newtheorem{corollary}[theorem]{Corollary}
\newtheorem{definition}{Definition}[section]
\newtheorem{example}{Example}[section]
\newtheorem{remark}{Remark}[section]

\newtheorem{assumption}{Assumption}

\title{The Power Method for   Non-Hermitian Dual Quaternion  Matrices}
\author{
Hao Yang\footnote{School of Mathematical Sciences, Beihang University, Beijing  100191, China.
			({\tt uyh@buaa.edu.cn})}
            \and
Liqun Qi\footnote{Department of Applied Mathematics, The Hong Kong Polytechnic University, Hung Hom, Kowloon, Hong Kong.
			({\tt maqilq@polyu.edu.hk})}
            \and 
Chunfeng Cui\footnote{School of Mathematical Sciences, Beihang University, Beijing  100191, China.
			({\tt chunfengcui@buaa.edu.cn})}
}
\date{\today}

\begin{document}
	\maketitle
	\begin{abstract}
        This paper proposes a power method for computing the dominant eigenvalues of a non-Hermitian dual quaternion matrix (DQM). Although the algorithmic framework parallels the Hermitian case, the theoretical analysis is substantially more complex since a non-Hermitian dual matrix may possess no eigenvalues or infinitely many eigenvalues. Besides, its eigenvalues are not necessarily dual numbers, leading to non-commutative behavior that further complicates the analysis.  
        We first present a sufficient condition that ensures the existence of an eigenvalue whose standard part corresponds to the largest magnitude eigenvalue of the standard part matrix.
        Under a stronger condition, we then establish that the sequence generated by the power method converges linearly to the strict dominant eigenvalue and its associated eigenvectors. We also verify that this condition is necessary. 
        The key to our analysis is a new Jordan-like decomposition, which addresses a gap arising from the lack of a conventional Jordan decomposition for non-Hermitian dual matrices. 
        Our framework readily extends to non-Hermitian dual complex and dual number matrices. We also develop an adjoint method that reformulates the eigenvalue problem into an equivalent form for dual complex matrices.
        Numerical experiments on non-Hermitian DQMs are presented to demonstrate the efficiency of the power method. \par
		

	\end{abstract}

\section{Introduction}
Dual quaternion numbers and   matrices have proven to be valuable in robotics,  
particularly in solving key problems such as  hand-eye calibration \cite{daniilidis1999hand,xie2025generalized,chen2025dual},  simultaneous localization and mapping \cite{bultmann2019stereo,bryson2007building,leonard2016past,carlone2015initialization,wei2013autonomous,williams2000autonomous}, and multi-agent formation control \cite{qi2025unit,qi2022dual,chen2025dual,Liu2025distributed}. In multi-agent formation control, the eigenvalues of DQMs play a crucial role in verifying the reasonableness of configurations \cite{qi2025unit,chen2025dual,qi2024eigenvalues} and establishing the convergence of control law \cite{cui2025DQMControl}.\par 

In recent years, the theoretical foundations and algorithmic design for the dual quaternion  eigenvalue problem have attracted significant research attentions. In \cite{qi2021eigenvalues}, Qi and Luo investigated the spectral theory, revealing that if a right eigenvalue of a DQM is a dual number, it is also a left eigenvalue. They showed that the right eigenvalues of a dual quaternion Hermitian matrix are dual numbers. An $n$-by-$n$  dual quaternion Hermitian matrix was shown to have $n$ eigenvalues. Thus, a unitary decomposition for a dual Hermitian matrix was proposed. In 2024, Cui and Qi \cite{cui2024power} introduced a power method for computing the strict dominant eigenvalues of dual quaternion Hermitian matrices. It proved that the convergence rate of the power method to the strict dominant eigenpairs of dual quaternion Hermitian matrices is linear. Subsequently,  the Rayleigh quotient iteration \cite{Duan2024} that achieves faster convergence rates was proposed.  
However, when the dual quaternion Hermitian matrix has two eigenvalues
with identical standard parts but different dual parts, the power method and the Rayleigh quotient may be invalid to calculate these eigenvalues.  
Several methods have proposed to address the above issue, such as the supplement matrix method \cite{qi2024eigenvalues}, the Jacobi method \cite{ding2024jacobi, chen2025three}, and EDDCAM-EA \cite{chen2025dual}.  


However, while significant advances have been made for Hermitian DQMs, to the best of our knowledge, computing  eigenvalues of non-Hermitian DQMs remains underdeveloped. Non-Hermitian DQMs naturally arise in applications involving weighted directed graphs \cite{qi2025unit} and formation control systems. In formation control problems, the underlying interaction graphs are typically digraphs to model asymmetric relationships between agents. This makes the computation of eigenvalues for non-Hermitian DQMs particularly important. However, as demonstrated in \cite{qi2023eigenvalues},  non-Hermitian dual complex matrices (DCMs) exhibit special spectral properties: they may possess no eigenvalues or infinitely many eigenvalues. This naturally extends to the dual quaternion case. These unique characteristics make non-Hermitian DQMs both theoretically significant and computationally challenging. 
   
\par 

In this paper, we propose computing the strict dominant eigenvalue of a non-Hermitian DQM by the power method. Although the algorithmic framework parallels that for the Hermitian case, the theoretical analysis is substantially more complex. A primary challenge arises from the uncertain existence of eigenvalues for non-Hermitian DQMs owing to their distinctive spectral characteristics. To address this, we establish a sufficient condition that guarantees the existence of an eigenvalue whose standard part corresponds to the largest magnitude eigenvalue of the standard part matrix.
Another key challenge in this task lies in computing the matrix power of non-Hermitian dual quaternions, which stems from their intrinsic non-commutativity and structural complexity. To tackle this challenge, we reduce a non-Hermitian DQM to a block diagonal matrix via similarity transformation, which is termed the Jordan-like decomposition. This enables us to analyze the properties of the power of a non-Hermitian DQM, which in turn provides a rigorous basis for analyzing the convergence behavior of the power method. Furthermore, to ensure the convergence of the power method, we introduce an assumption regarding the spectral properties of the matrix. We also demonstrate that the power method may fail to converge if this assumption is not satisfied.

The structure of this paper is organized as follows. Section 2 introduces some fundamental knowledge about quaternions and dual quaternions. Section 3 shows a sufficient condition for the existence of the dominant eigenvalue and presents the Jordan-like decomposition of a DQM. On this basis, we establish the convergence, convergence rate, and necessary and sufficient conditions for the convergence of the power method applied to non-Hermitian DQMs. Furthermore, we extend the conclusions to the power method for non-Hermitian  DCMs and develop the dual complex adjoint matrix (DCAM) power method for computing the eigenvalues of non-Hermitian DQMs. 
Section 4 presents numerical experiment results to demonstrate the efficiency and properties of the power method and the  DCAM  based power method.

\section{Preliminaries}
In this section, we introduce some basic knowledge about the quaternion, quaternion vector, quaternion matrix, dual quaternion, dual quaternion vector, and dual quaternion matrix (DQM). 
The set of real numbers, dual numbers, complex numbers, dual complex numbers, quaternions, dual quaternions are denoted as $\mathbb{R},\mathbb{DR},\mathbb{C},\mathbb{DC},\mathbb{Q},\mathbb{DQ}$, respectively.

\subsection{Quaternions}

A quaternion $q\in\mathbb{Q}$ is expressed as $q=q_0+q_1\ii+q_2\jj+q_3\kk$, where $q_0,q_1,q_2,q_3\in\mathbb{R}$, and three imaginary units $\ii,\jj,\kk$ satisfy
\begin{equation*}
	\ii^2=\jj^2=\kk^2=\ii\jj\kk=-1.
\end{equation*}
The quaternion multiplication is associative but non-commutative. The conjugate and magnitude of a quaternion $q$ is defined as $q^*=q_0-q_1\ii-q_2\jj-q_3\kk$ and  $|q|=q^*q=\sqrt{q_0^2+q_1^2+q_2^2+q_3^2}$,  
respectively. 
 Every non-zero quaternion $q$ is invertible, and its inverse is $q^{-1}=\frac{q^*}{|q|^2}$. If $|q|=1$, then it is called a unit quaternion.

Two quaternions $q$ and $p$ are said to be similar and are denoted as $p\sim q$ if there exists an invertible quaternion $u$ such that $u^{-1}qu=p$. The similarity relationship of quaternions is an equivalent relationship. We denote by $[q]$ the equivalent class containing $q$.
\begin{lemma}
	[\cite{zhang1997quaternions}]
	If $q=q_0+q_1\ii+q_2\jj+q_3\kk\in\mathbb{Q}$, then $q_0+\sqrt{q_1^2+q_2^2+q_3^2}\ii\in \mathbb C$ and $q$ are similar. Namely, $q\in[q_0+\sqrt{q_1^2+q_2^2+q_3^2}\ii]$.
\end{lemma}


The set of $n$-dimensional quaternion vectors is denoted as $\mathbb{Q}^n$. 
The 2-norm of a quaternion vector $\vx=(x_i)\in\mathbb{Q}^{n}$ is defined as $\Vert \vx\Vert_2=\sqrt{\sum_{i=1}^n|x_i|^2}$. For $\vx,\vy\in\mathbb{Q}^n$, 
\begin{equation}\label{eq:dual}
    \vx^*\vy+\vy^*\vx\leq 2\Vert\vx\Vert_2\Vert\vy\Vert_2.
\end{equation}
	We say that $\vx_1,\vx_2,\dots,\vx_r\in\mathbb{Q}^n$ are right linearly independent, if for any quaternion numbers $k_1,k_2,\dots,k_r\in\mathbb{Q}$, it holds that 
	\begin{equation*}
		\vx_1k_1+\vx_2k_2+\cdots+\vx_rk_r=0 \text{ implies }k_1=k_2=\cdots=k_r=0. 
	\end{equation*}
	Otherwise, we say that $\vx_1,\vx_2,\dots,\vx_r$ are right linearly dependent. 

The set of $m$-by-$n$ quaternion matrices is denoted as $\mathbb{Q}^{m\times n}$. Suppose that $A\in\mathbb{Q}^{n\times n}$ and $B\in\mathbb{Q}^{n\times n}$ are two quaternion matrices. If $AB=BA=I$, where $I$ is the $n\times n$ identity matrix, then we say $A$ is invertible, $B$ is the inverse of $A$, and denote $B=A^{-1}$. The conjugate of a quaternion matrix $A=(a_{ij})\in\mathbb{Q}^{m\times n}$ is denoted as $\overline{A}=(a_{ij}^*)\in\mathbb{Q}^{m\times n}$.
The conjugate transpose of $A$ is defined as $A^*=(a_{ji}^*)\in\mathbb{Q}^{n\times m}$. 
A quaternion $\lambda$ is said to be a right eigenvalue of $A\in\mathbb{Q}^{n\times n}$ if $A\vx=\vx\lambda$ for a non-zero quaternion vector $\vx\in\mathbb{Q}^n$, and $\vx$ is the corresponding eigenvector. For convenience, we denote right eigenvalues as eigenvalues. If $\lambda\in\mathbb{Q}$ is an eigenvalue of $A\in\mathbb{Q}^{n\times n}$ with the corresponding eigenvector $\vx$, then $\alpha^*\lambda \alpha\in[\lambda]$ is also an eigenvalue of $A$ with the corresponding eigenvector $\vx\alpha$, for any unit quaternion $\alpha$. We call $\lambda$ a standard eigenvalue of $A$ if $\lambda$ is a complex number with nonnegative imaginary part. Any $n$-by-$n$ quaternion matrix has at least one standard eigenvalue~\cite{brenner1951matrices}. Furthermore, any $n$-by-$n$ quaternion matrix has the Jordan form as follows \cite{zhang1997quaternions}.

\begin{theorem}
	[\cite{zhang1997quaternions}] \label{Thm:Jordan_Q}
	Suppose that $\lambda_1,\dots,\lambda_m\in\mathbb C$ are distinct standard eigenvalues of $A\in\mathbb{Q}^{n\times n}$, in which the algebraic 
    multiplicities are $n_1,\dots,n_m$, respectively. Then there exist an invertible matrix $P\in\mathbb{Q}^{n\times n}$ such that 
	\begin{equation*}
		P^{-1}AP=J=\diag(J_1,\dots,J_m),
	\end{equation*}
	in which  $J_i=\diag(J_{n_{i1}}(\lambda_i),J_{n_{i2}}(\lambda_i),\dots,J_{n_{it_i}}(\lambda_i))\in\mathbb{Q}^{n_i\times n_i},$ 
	\begin{equation*}
		J_{n_{ik}}(\lambda_i)=
		\begin{pmatrix}
			\lambda_i&1&&\\
			&\ddots&\ddots&\\
			&&\ddots&1\\
			&&&\lambda_i
		\end{pmatrix}\in\mathbb{Q}^{n_{ik}\times n_{ik}}, 1\leq k\leq t_i, 
		\sum_{k=1}^{t_i}n_{ik}=n_i, 1\leq i\leq m,
	\end{equation*}
	and except the order of $J_{n_{ik}}(\lambda_i)$, $J$ is unique.  
\end{theorem}
The matrices $J_{n_{ik}}(\lambda_i)(1\leq k\leq t_i, 1\leq i\leq m)$ are called Jordan blocks, and the matrix $J$ is called the Jordan canonical form of $A$. 
The algebraic multiplicity of an eigenvalue $\lambda$, denoted by $\mathtt{m}_a(\lambda, A)$, is the sum of the dimensions of all Jordan blocks associated with $\lambda$. Its geometric multiplicity, denoted $\mathtt{m}_g(\lambda, A)$, is the number of such Jordan blocks. 
It is straightforward to observe that $$\mathtt{m}_a(\lambda, A) \ge \mathtt{m}_g(\lambda, A)$$ for any eigenvalue $\lambda$. Moreover, a matrix $A$ is diagonalizable if and only if $\mathtt{m}_a(\lambda, A) = \mathtt{m}_g(\lambda, A)$ holds for every eigenvalue $\lambda$.

\par
For $A=(a_{ij})\in\mathbb{Q}^{m\times n}$, the 1-norm, 2-norm, $\infty$-norm, and the F-norm are defined by 
\begin{equation*}
	\begin{aligned}
		\Vert A\Vert_1&=\max_{1\leq j\leq n}\sum_{i=1}^m|a_{ij}|,\ \ \Vert A\Vert_2=\sqrt{\lambda_{\max}(A^*A)},\\
		\Vert A\Vert_{\infty}&=\max_{1\leq i\leq m}\sum_{j=1}^n|a_{ij}|,\ \ \Vert A\Vert_F=\sqrt{\sum_{i,j}|a_{ij}|^2},
	\end{aligned}
\end{equation*}
where $\lambda_{\max}(A^*A)$ is the largest eigenvalue of $A^*A$. 
Suppose that $A\in\mathbb{Q}^{m\times n},B\in\mathbb{Q}^{n\times r}$, $\vx\in\mathbb{Q}^{n}$,  the following   norm inequalities hold \cite{wei2018quaternion}:
\begin{equation*}
	\Vert A\vx\Vert_2\leq \Vert A\Vert_2\Vert \vx\Vert_2,\ \ 
	\Vert AB\Vert_2\leq \Vert A\Vert_2\Vert B\Vert_2,\ \
	\text{and }\Vert A\Vert_2^2\leq\Vert A\Vert_1\Vert A\Vert_{\infty}.
\end{equation*}

\subsection{Dual quaternions}

A
dual quaternion $\hat{q}=q_s+q_d\epsilon\in\mathbb{DQ}$ has the standard part $q_s\in\mathbb Q$ and the dual part $q_d\in\mathbb Q$. The symbol $\epsilon$ is the infinitesimal unit, satisfying $\epsilon\neq 0$, $\epsilon^2=0$, and $\epsilon$ is commutative with quaternions. If $q_s\neq 0$, then we say that $q$ is appreciable. If $q_s$ and $q_d$ are real numbers, then $\hat{q}$ is called a dual number and $\hat{q}\in\mathbb{DR}$. Similarly, dual complex numbers can   be defined.
Throughout this paper, we  denote st($\hat{q}$)  and  du($\hat{q}$) as the standard and  dual parts of $\hat q$,  respectively. The conjugate of $\hat{q}=q_s+q_d\epsilon\in\mathbb{DQ}$ is $\hat{q}^*=q_s^*+q_d^*\epsilon$, where $q_s^*$ and $q_d^*$ are the conjugates of $q_s$ and $q_d$, respectively. 
Given two dual quaternions $\hat{p}=p_s+p_d\epsilon$ and $\hat{q}=q_s+q_d\epsilon$, their sum is $\hat{p}+\hat{q}=(p_s+q_s)+(p_d+q_d)\epsilon$, and their product is equal to $\hat{p}\hat{q}=p_sq_s+(p_sq_d+p_dq_s)\epsilon$.

For two dual numbers $\hat{a}=a_s+a_d\epsilon, \hat{b}=b_s+b_d\epsilon\in\mathbb{DR}$,
the division of $\hat{a}$ and $\hat{b}$, when $a_s\neq 0$, or $a_s=0$ and $b_s=0$ is defined by 
\begin{equation*}
	\frac{b_s+b_d\epsilon}{a_s+a_d\epsilon}=
	\begin{cases}
		\frac{b_s}{a_s}+\left(\frac{b_d}{a_s}-\frac{b_s}{a_s}\frac{a_d}{a_s}\right)\epsilon, &\text{if }a_s\neq 0,\\
		\frac{b_d}{a_d}+c\epsilon, &\text{if }a_s=0,b_s=0,
	\end{cases}
\end{equation*}
where $c\in\mathbb{R}$ is an arbitrary real number. 
If $a_s>b_s$ or $a_s=b_s$ and $a_d>b_d$, we say $\hat{a}>\hat{b}$. This defines positive and  nonnegative dual numbers \cite{qi2022dual}. 
\begin{definition}
	[\cite{cui2024power}]\label{def:O}
	Let $\{\hat{a}_k=a_{ks}+a_{kd}\epsilon:k=1,2,\dots\}$ be a dual number sequence. We say that the dual number sequence $\{\hat{a}_k\}$ converges to a dual number $\hat{a}=a_s+a_d\epsilon$, if  $\{a_{ks}\}$ converges to $a_s$ and  $\{a_{kd}\}$ converges to $a_d$, respectively.\par 
	Let $\{c_k:k=1,2,\dots\}$ be a real number sequence. We denote by $\hat{a}_k=O_D(c_k)$ if $a_{ks}=O(c_k)$ and $a_{kd}=O(c_k)$. Furthermore, if $c_k=O(c^kh(k))$ for a real number $0<c<1$ and a polynomial $h(k)$, then we denote $c_k=\tilde{O}(c^k)$. Similarly, $\hat{a}_k=\tilde{O}_D(c^k)$ if $a_{ks}=\tilde{O}(c^k)$ and $a_{kd}=\tilde{O}(c^k)$. 
\end{definition}

The magnitude of a dual quaternion $\hat{q}=q_s+q_d\epsilon$ is defined as a nonnegative dual number
\begin{equation*}
	|\hat{q}|:=
	\begin{cases}
		|q_s|+\frac{(q_s^*q_d+q_d^*q_s)}{2|q_s|}\epsilon,\ & \text{if $q_s\neq 0$},\\
		|q_d|\epsilon,&\text{otherwise}.
	\end{cases}
\end{equation*} 
Every non-zero appreciable dual quaternion $\hat{q}$ is invertible, and its inverse is $\hat{q}^{-1}=\frac{\hat{q}^*}{|\hat{q}|^2}$. If $|\hat{q}|=1$, then it is called a unit dual quaternion. 
Two dual quaternions $\hat{p}$ and $\hat{q}$ are similar if and only if there is a unit dual quaternion $\hat{v}$ such that $\hat{v}^{*}\hat{p}\hat{v}=\hat{q}$. 
We denote by $[\hat{q}]_D$ the equivalent class containing $\hat{q}$.

\begin{lemma}
	[\cite{cui2024spectral,chen2025dual}]
	If $\hat{q}=q_s+q_d\epsilon\in\mathbb{DQ}$, then there exists a dual complex number $\hat{p}=p_{s}+p_d\epsilon\in\mathbb{DC}$ such that $\hat{p}\in[\hat{q}]_D$. Furthermore, there is $\textup{Re}(\hat{q})=\textup{Re}(\hat{p})$ and $|\textup{Im}(\hat{q})|=|\textup{Im}(\hat{p})|$. If $\hat{r}\in\mathbb{DC}\cap[\hat{q}]_D$, then $\hat{r}=\hat{p}$ or $\hat{r}=\hat{p}^*$.
\end{lemma}

The $2$-norm and the $2^R$-norm of a dual quaternion vector $\hat{\vx}=\vx_s+\vx_d\epsilon\in\mathbb{DQ}^n$ are 
\begin{equation*}
	\Vert \hat{\vx}\Vert_2=
	\begin{cases}
		\Vert\vx_s\Vert_2+\frac{\vx_s^*\vx_d+\vx_d^*\vx_s}{2\Vert\vx_s\Vert_2}\epsilon, &\text{if }\vx_s\neq 0,\\
		\Vert\vx_d\Vert_2\epsilon, &\text{if }\vx_s=0\text{ and }\hat{\vx}=\vx_d\epsilon,
	\end{cases}
\end{equation*}
 and $\Vert \hat{\vx}\Vert_{2^R}=\sqrt{\Vert \vx_s\Vert_2^2+\Vert \vx_d\Vert_2^2}$, respectively. 
Let $\hat{\vx}_i=\vx_{is}+\vx_{id}\epsilon$ for $i=1,\dots,k$ be dual quaternion vectors. If $\vx_{1s},\vx_{2s},\dots,\vx_{ks}$ are right linearly independent, then we say that $\hat{\vx}_1,\hat{\vx}_2,\dots,\hat{\vx}_k$ are appreciably linearly independent.

A dual quaternion matrix (DQM) $\hat{A}\in\mathbb{DQ}^{m\times n}$ has the standard part $A_s\in\mathbb{Q}^{m\times n}$ and the dual part $A_d\in\mathbb{Q}^{m\times n}$.  
Suppose that $\hat{A}=A_s+A_d\epsilon$ and $\hat{B}=B_s+B_d\epsilon$ are two DQMs. If $\hat{A}\hat{B}=\hat{B}\hat{A}=I$, where $I$ is the $n\times n$ identity matrix, then we say $\hat{A}$ is invertible, $\hat{B}$ is the inverse of $\hat{A}$, and denote $\hat{B}=\hat{A}^{-1}$. Here we have $B_s=A_s^{-1}$ and $B_d=-A_s^{-1}A_dA_s^{-1}$. 
For a DQM $\hat{A}=(\hat{a}_{ij})\in\mathbb{DQ}^{m\times n}$, denote its conjugate transpose as $\hat{A}^*=(\hat{a}_{ji}^*)\in\mathbb{DQ}^{n\times m}$. If $\hat{A}^*=\hat{A}$, then $\hat{A}$ is called a Hermitian DQM, otherwise $\hat{A}$ is called a non-Hermitian DQM. 
The $F$-norm of a DQM $\hat{A}$ is
\begin{equation*}
	\Vert \hat{A}\Vert_{F}=
	\begin{cases}
		\Vert A_s\Vert_F+\frac{tr(A_s^*A_d+A_d^*A_s)}{\Vert A_d\Vert_F}\epsilon,\ & \text{if $A_s\neq 0$},\\
		\Vert A_d\Vert_F\epsilon, &\text{otherwise}.
	\end{cases}
\end{equation*}
This is a dual number. The $F^R$-norm of a DQM is  $\Vert \hat{A}\Vert_{F^R}=\sqrt{\Vert A_s\Vert_F^2+\Vert A_d\Vert_F^2}.$ 
It should be noted that the $2^R$-norm and the $F^R$-norm are not norms since they do not satisfy the scaling condition of the norms \cite{cui2024power}. 

Let $\hat{A}\in\mathbb{DQ}^{n\times n}$, $\hat{\vx}\in\mathbb{DQ}^n$ be appreciable, and $\hat{\lambda}\in\mathbb{DQ}$. If
\begin{equation}\label{eq:eig}
	\hat{A}\hat{\vx}=\hat{\vx}\hat{\lambda},
\end{equation}
then $\hat{\lambda}$ is called a right eigenvalue of $\hat{A}$, with $\hat{\vx}$ as its corresponding right eigenvector. 
Expanding the standard part and dual part of equation~\eqref{eq:eig} yields 
    \begin{equation*} 
    \begin{cases}
         A_s\vx_s=\vx_s\lambda_{s},\\
         A_s\vx_d+A_d\vx_s=\vx_d\lambda_{s}+\vx_s\lambda_d.
    \end{cases} 
    \end{equation*}
If $\hat{\lambda}\in\mathbb{DQ}$ is an eigenvalue of $\hat{A}\in\mathbb{DQ}^{n\times n}$ with the corresponding eigenvector $\hat{\vx}$, then every $\hat{\alpha}^*\hat{\lambda}\hat{\alpha}\in[\lambda]_D$ is an eigenvalue of $\hat{A}$ with the corresponding eigenvector $\hat{\vx}\hat{\alpha}$, where $\hat{\alpha}$ is an arbitrary  unit dual quaternion. 

     
\subsection{Dual complex adjoint matrix (DCAM)}
In this subsection, we introduce DCAM and the connection between the eigenvalue problems of DQMs and DCAMs.

	Let $\hat{Q}=(A_1+A_2\jj)+(A_3+A_4\jj)\epsilon\in\mathbb{DQ}^{m\times n}$, $A_1,A_2,A_3,A_4\in\mathbb{C}^{m\times n}$. The DCAM \cite{chen2025dual} of  $\hat{Q}$ is defined by
	\begin{equation}
		\mathcal{J}(\hat{Q})=
		\begin{pmatrix}
			A_1&A_2\\
			-\overline{A_2}&\overline{A_1}
		\end{pmatrix}+
        \begin{pmatrix}
			A_3&A_4\\
			-\overline{A_4}&\overline{A_3}
		\end{pmatrix}\epsilon\in\mathbb{DC}^{2m\times 2n}.
	\end{equation} 
	Let $\hat{\vv}=(\vv_1+\vv_2\jj)+(\vv_3+\vv_4\jj)\epsilon\in\mathbb{DQ}^{n}$ and  $\mathcal{F}(\hat{\vv})=
	\begin{pmatrix}
		\vv_1\\
		-\overline{\vv}_2
	\end{pmatrix}+
    \begin{pmatrix}
		\vv_3\\
		-\overline{\vv}_4
	\end{pmatrix}\epsilon$, then $\mathcal{J}(\hat{\vv})=
	\begin{pmatrix}
		\mathcal{F}(\hat{\vv})&-\mathcal{F}(\hat{\vv}\jj)
	\end{pmatrix}$. $\mathcal{F}$ is bijection with inverse $\mathcal{F}^{-1}: \mathbb{DC}^{2n}\rightarrow\mathbb{DQ}^{n}$ given by $$\mathcal{F}^{-1}\left(\begin{pmatrix}
	    \vu_1\\
        \vu_2
	\end{pmatrix}+\begin{pmatrix}
	    \vu_3\\
        \vu_4
	\end{pmatrix}\epsilon\right)=\vu_1-\overline{\vu}_2\jj+(\vu_3-\overline{\vu}_4\jj)\epsilon.$$

	\begin{theorem}[\cite{chen2025dual}]\label{Thm:adjoint}
		Let $\hat{Q}\in\mathbb{DQ}^{n\times n}$, $\hat{\vv}\in\mathbb{DQ}^{n\times 1}$, $\hat{\lambda}\in\mathbb{DC}$, $P=\mathcal{J}(\hat{Q}), \vu_1=\mathcal{F}(\hat{\vv})$, and $ \vu_2=\mathcal{F}(\hat{\vv})$, then 
		\begin{equation*}
			\hat{Q}\hat{\vv}=\hat{\vv}\hat{\lambda} \quad \Longleftrightarrow\quad P\vu_1=\hat{\lambda}\vu_1\ or \ P\vu_2=\hat{\lambda}^*\vu_2.
		\end{equation*}
		In addition, $\vu_1$ and $\vu_2$ are orthogonal.
	\end{theorem}
 
\section{The power method for  non-Hermitian dual quaternion matrices}

For a quaternion matrix or a Hermitian DQM, the power method can compute its eigenvalues with the maximum absolute value \cite{li2019power,jia2023computing,cui2024power}.
In this section, we extend the power method of Hermitian DQMs \cite{cui2024power} to the non-Hermitian counterparts, as detailed in Algorithm~\ref{alg_PM}. Despite an identical framework, the theory differs fundamentally.  This is because the eigenvalues $\hat{\lambda}$ of non-Hermitian DQMs are non-commutative dual quaternions, which invalidates the well-established theory for Hermitian matrices, including guarantees of eigenvalue existence and the availability of a Jordan canonical form.

The computational burden of Algorithm~\ref{alg_PM} is dominated by line 2, which requires $O(n^2)$ flops. All other steps execute in $O(n)$ flops. Therefore, each iteration step of Algorithm~\ref{alg_PM} has an overall computational complexity of $O(n^2)$ flops.

\subsection{Existence of eigenvalues}

Given that $\hat A$ may either possess no eigenvalues or infinitely many eigenvalues \cite{qi2023eigenvalues}, we introduce the following assumption to ensure the existence of eigenvalues whose standard part is $\lambda_{1s}$.
\par

\begin{algorithm}[t]
	\caption{The power method (PM) for computing the dominant eigenvalue of  non-Hermitian DQMs.}\label{alg_PM} 
	\begin{algorithmic}[1]
		\Require a non-Hermitian DQM $\hat{A}\in\mathbb{DQ}^{n\times n}$, an initial vector $\hat{\vv}^{(0)}\in\mathbb{DQ}^n$, the maximal iteration number $k_{\max}$ and the tolerance $\delta$.
		\For{$k=1,2,\dots,k_{\max}$}
		\State $\hat{\vy}^{(k)}=\hat{A}\hat{\vv}^{(k-1)}$,
		\State $\hat{\lambda}^{(k-1)}=(\hat{\vv}^{(k-1)})^*\hat{\vy}^{(k)}$.
		\If{$\Vert \hat{\vy}^{(k)}-\hat{\vv}^{(k-1)}\hat{\lambda}^{(k-1)}\Vert_{2^R}\leq \delta$}
        \State break.
		\EndIf
		\State $\hat{\vv}^{(k)}=\frac{\hat{\vy}^{(k)}}{\Vert \hat{\vy}^{(k)}\Vert_2}$.
		\EndFor
		\Ensure $\hat{\vv}^{(k-1)}$, $\hat{\lambda}^{(k)}$.
	\end{algorithmic}
\end{algorithm}

\begin{assumption}\label{assump}  
 Let $\hat{A}=A_s+A_d\epsilon\in\mathbb{DQ}^{n\times n}$, and  let $\lambda_{1s},\dots,\lambda_{ms}\in\mathbb{C}$ be distinct standard eigenvalues of $A_s$. We assume that   
\begin{equation}\label{eq:eigen} |\lambda_{1s}|>|\lambda_{2s}|\geq\cdots\geq|\lambda_{ms}| \text{  and   } \mathtt{m}_a(\lambda_{1s},A_s)=\mathtt{m}_g(\lambda_{1s},A_s). 
\end{equation}  
\end{assumption} 
 In the following theorem, we demonstrate that Assumption~\ref{assump} serves as a sufficient condition for the existence of eigenvalues whose standard part is $\lambda_{1s}$. 
 It should be noted, however, that this assumption is not necessary, as illustrated by 
 $\begin{pmatrix}
     2+\epsilon&1&0\\
     0&2&0\\
     0&0&1
 \end{pmatrix}$. Despite violating Assumption~\ref{assump}, it possesses an eigenvalue $2+\alpha\epsilon$ whose standard part equals $\lambda_{1s}=2$, and the corresponding eigenvector $[1+(\alpha-1)\epsilon,0,0]^T$ , where $\alpha$ can be any quaternion. 

 We first state a lemma regarding dual complex matrices, which serves as the foundation for the sufficiency of Assumption~\ref{assump} for DQMs. 
\begin{lemma}\label{Lem:complex}
   Suppose $\hat{B}=B_s+B_d\epsilon \in \mathbb{DC}^{n \times n}$ is a dual complex matrix and $\lambda_{1s}\in\mathbb{C}$ is an eigenvalue of $B_s$. If $\mathtt{m}_a(\lambda_{1s},B_s)=\mathtt{m}_g(\lambda_{1s},B_s)$, then $\hat{B}$ has at least one eigenvalue with standard part equal to $\lambda_{1s}$. Moreover, the number of such eigenvalues is finite. 
\end{lemma}
\begin{proof}
 Given that $\mathtt{m}_a(\lambda_{1s},B_s)=\mathtt{m}_g(\lambda_{1s},B_s)$, we may assume without loss of generality that 
    \begin{equation*}
        B_s=\begin{pmatrix}
        \lambda_{1s}I_{n_1}&\\
        &C
        \end{pmatrix},\   B_d=
        \begin{pmatrix}
           D&E\\
           F&G
        \end{pmatrix},
    \end{equation*}
    where $C\in\mathbb{C}^{(n-n_1)\times (n-n_1)}$ and none of the eigenvalues of $C$
    equals $\lambda_{1s}$,   
    $D\in\mathbb{C}^{n_1\times n_1}$, $F\in\mathbb{C}^{(n-n_1)\times n_1}, E\in\mathbb{C}^{n_1\times (n-n_1)}$, and $G\in\mathbb{C}^{(n-n_1)\times (n-n_1)}$.   
   Let $\lambda_s=\lambda_{1s}$, $\vx_s=\begin{pmatrix}
        \mathbf{y}\\
        \mathbf{0}
    \end{pmatrix}\in\mathbb{C}^{n}$, where $\mathbf{y}\in\mathbb C^{n_1}$ is a unit right eigenvector of $D$ corresponding to the eigenvalue $\lambda_d=\mathbf{y}^*D\mathbf{y}$, and  
      $\vx_d=\begin{pmatrix}
        \mathbf{0}\\
        \mathbf{z}
    \end{pmatrix}$, where $\mathbf{z}\in\mathbb C^{n-n_1}$. 
    Then we may check that $\lambda=\lambda_s+\lambda_d\epsilon$ and $\vx=\vx_s+\vx_d\epsilon$ is an eigenpair of $\hat B$ as long as $\mathbf{z}$ satisfies the equation $(C-\lambda_{1s}I_{n-n_1})\mathbf{z}=-F\mathbf{y}$. The equation has a unique solution, since none of the eigenvalues of $C$ equals $\lambda_{1s}$. 

    If $\lambda_{d}$ is not an eigenvalue of $D$, then the eigenvalue equation $B_s\vx_d+B_d\vx_s=\lambda_{1s}\vx_d+\lambda_d\vx_s$ cannot hold. Consequently, the number of eigenvalues whose standard part equals $\lambda_{1s}$ coincides with the number of eigenvalues of $D$, which is finite. 
    This completes the proof.  
\end{proof}

We now utilize Lemma~\ref{Lem:complex} to prove the existence of eigenvalues for DQMs under Assumption~\ref{assump}.
\begin{theorem}
    Suppose that $\hat{A}\in\mathbb{DQ}^{n\times n}$ satisfies Assumption~\ref{assump}. Then $\hat{A}$ has at least one eigenvalue whose standard part is $\lambda_{1s}$. Furthermore, the number of eigenvalues whose standard part lies in $[\lambda_{1s}]$ is finite.
\end{theorem}
\begin{proof}
    By Theorem \ref{Thm:adjoint}, $\lambda_{1s}$ is an eigenvalue of $\mathcal{J}(A_s)\in\mathbb{C}^{2n\times 2n}$, with $\mathtt{m}_a(\lambda_{1s},A_s)=\mathtt{m}_g(\lambda_{1s},A_s)$. Lemma~\ref{Lem:complex} further implies that $\mathcal{J}(\hat{A})$ has at least one eigenvalue whose standard part is $\lambda_{1s}$, and the number of such eigenvalues is finite. Applying Theorem \ref{Thm:adjoint} once more, we conclude that $\hat{A}$ also has at least one eigenvalue whose standard part lies in $[\lambda_{1s}]$, and the number of such eigenvalues remains finite.
\end{proof}

However, the total number of eigenvalues may be infinite, as illustrated in the following example. 
\begin{example}
    A non-Hermitian DQM $\hat{A}\in\mathbb{DQ}^{n\times n}$ satisfies Assumption~\ref{assump}, but it has infinitely many eigenvalues.  
    Let $\hat{A}=A_s+A_d\epsilon$, where
    \begin{equation*}
        A_s=\begin{pmatrix}
            2&0&0\\
            0&1&1\\
            0&0&1
        \end{pmatrix} \text{ and } 
        A_d=\begin{pmatrix}
            1&0&0\\
            0&1&0\\
            0&0&1
        \end{pmatrix}.
    \end{equation*}
    By direct computations, the eigenvalues of $\hat{A}$ are $2+\epsilon,1+\alpha\epsilon$, with corresponding eigenvectors $[1,0,0]^\top$ and $[0,1+(\alpha+1)\epsilon,0]^\top$ respectively, where $\alpha$ can be any quaternion. Thus, $\hat{A}$ has infinitely many eigenvalues. 
\end{example}


\subsection{The Jordan-like form}

A crucial step in the convergence analysis of Algorithm~\ref{alg_PM} involves understanding powers of DQMs. Recall that for any quaternion matrix $Q$, its $k$-th power $Q^k$ can be expressed using its Jordan canonical form. Specifically, if  $Q=P^{-1}JP$, where $P$ is an invertible quaternion matrix and $J$ is Jordan form of $Q$, then $Q^k=P^{-1}J^kP$. 
Similar properties extend to certain classes of dual number matrices including Hermitian DQMs~\cite{qi2021eigenvalues}. Moreover, for non-Hermitian dual complex matrices whose standard part is a Jordan block, their Jordan structure has been studied in~\cite{qi2023eigenvalues}. Nevertheless, these existing results do not directly generalize to general non-Hermitian DQMs. To facilitate the convergence analysis of Algorithm~\ref{assump}, we introduce a Jordan-like canonical form tailored to this broader class of matrices.

\begin{theorem}[Jordan-like form]\label{Thm:Jordan like}
Suppose $\hat{A} = A_s + A_d\epsilon \in \mathbb{DQ}^{n \times n}$, and $A_s$ has standard eigenvalues $\lambda_{1s}, \lambda_{2s}, \dots,$ $\lambda_{ms}$ $\in \mathbb{C}$. Let $P_s \in \mathbb{Q}^{n \times n}$ and the Jordan blocks $\{J_i\}_{i=1}^m$ be defined as in Theorem~\ref{Thm:Jordan_Q}. Then there exists a DQM $\hat{P} = P_s + P_d\epsilon \in \mathbb{DQ}^{n \times n}$ and matrices $H_i \in \mathbb{Q}^{n_i \times n_i}$ such that $\hat{P}^{-1}\hat{A}\hat{P}$ takes the following block diagonal Jordan-like form:  
	\begin{equation}\label{eq:diag}
		\hat{P}^{-1}\hat{A}\hat{P}=
		\begin{pmatrix}
			J_1+H_1\epsilon&&&\\
			&J_2+H_2\epsilon&&\\
			&&\ddots&\\
			&&&J_m+H_m\epsilon
		\end{pmatrix}.
	\end{equation} 
\end{theorem}
\begin{proof}
    By direct computations, we obtain that 
    \begin{equation}\label{eq:proof3.1}
    \begin{aligned}
        \hat{P}^{-1}\hat{A}\hat{P}&=(P_s^{-1}-P_s^{-1}P_dP_s^{-1}\epsilon)(A_s+A_d\epsilon)(P_s+P_d\epsilon)\\
        &=P_s^{-1}A_sP_s+(P_s^{-1}A_dP_s+P_s^{-1}A_sP_sD-DP_s^{-1}A_sP_s)\epsilon, 
    \end{aligned}
    \end{equation}
    where $D=P_s^{-1}P_d$ in the last equality. Since the standard part of $\hat{P}^{-1}\hat{A}\hat{P}$ is already consistent with the form of  \eqref{eq:diag}, we only need to consider the dual part  of \eqref{eq:diag}.  To this end, the key lies in selecting an appropriate $P_d$, which is equivalent to choosing a suitable $D$ since $P_s$ is an invertible matrix.
    
    For notational convenience, denote  $P_s=(V_{1}\ V_2\ \cdots\ V_m)$, $(P_s^{-1})^*=(W_{1}\ W_{2}\ \cdots\ W_{m})$, where $V_i,W_i\in\mathbb{Q}^{n\times n_i}$.
    Furthermore, let $D=(D_{ij})$, 	where each block $D_{ij}\in\mathbb{Q}^{n_i\times n_j}$. Combining with $P_s^{-1}A_sP_s=\diag(J_1,\dots,J_m)$, we obtain the dual part of $\hat{P}^{-1}\hat{A}\hat{P}$ as follows, 
	\begin{equation*}
		\begin{aligned}
			&P_s^{-1}A_dP_s+P_s^{-1}A_sP_sD-DP_s^{-1}A_sP_s\\
			=&P_s^{-1}A_dP_s+
			\begin{pmatrix}
					J_1D_{11}-D_{11}J_1&J_1D_{12}-D_{12}J_2&\cdots&J_1D_{1m}-D_{1m}J_m\\
					J_2D_{21}-D_{21}J_1&J_2D_{22}-D_{22}J_2&\cdots&J_2D_{2m}-D_{2m}J_m\\
					\vdots&\vdots&\ddots&\vdots\\
					J_mD_{m1}-D_{m1}J_1&J_mD_{m2}-D_{m2}J_2&\cdots&J_mD_{mm}-D_{mm}J_m
			\end{pmatrix}\\
			=&
				\begin{pmatrix}
					K_{11}&K_{12}&\cdots&K_{1m}\\
					K_{21}&K_{22}&\cdots&K_{2m}\\
					\vdots&\vdots&\ddots&\vdots\\
					K_{m1}&K_{m2}&\cdots&K_{mm}
			\end{pmatrix},
		\end{aligned}
	\end{equation*}
    where $K_{ij}=W_i^*A_dV_j+J_iD_{ij}-D_{ij}J_j$.

	For any $i,j\in\{ 1,\dots,m\}$, the equation 
    $$K_{ij}=W_i^*A_dV_j+J_iD_{ij}-D_{ij}J_j=0$$ 
    is a Sylvester equation in $D_{ij}$. According to \cite[Theorem 5.11.1] {rodman2014topics},  a solution $D_{ij}$ exists whenever $i\neq j$, since the right eigenvalues of $J_i$ and $J_j$ are distinct. For $i=j$, we   set $D_{ii}=I_{n_i}$ and $H_i=W_i^*A_dV_i$.
    Consequently, 
	\begin{equation*}
		\begin{aligned}
			\hat{P}^{-1}\hat{A}\hat{P}
			&=P_s^{-1}AP_s+(P_s^{-1}A_dP_s+P_s^{-1}A_sP_sD-DP_s^{-1}A_sP_s)\epsilon.\\
			&=\diag(J_1+H_1\epsilon,J_2+H_2\epsilon,\cdots,J_m+H_m\epsilon).
		\end{aligned}
	\end{equation*} 
	This completes the proof.
\end{proof}

 It is important to note that the Sylvester equations are solely employed for theoretical purposes in  
 computing $P_d$,  and are not required in the practical implementation of Algorithm~\ref{alg_PM}. Indeed, Theorem~\ref{Thm:Jordan like} is the key for expressing the power of DQMs in our analysis. Namely,  
 \begin{equation}
    A^k = P\text{diag}((J_1+H_1\epsilon)^k,\dots,(J_m+H_m\epsilon)^k)P^{-1}. 
 \end{equation}

\begin{remark}
Unlike the Jordan form of a real matrix, the Jordan-like form of a DQM is not unique.
In the construction outlined in Theorem~\ref{Thm:Jordan like}, the block $D_{ii}$ can be chosen arbitrarily.  $\hat{P}^{-1}\hat{A}\hat{P}$ still retains the block diagonal form given in \eqref{eq:diag} with $H_i=W_i^*A_dV_i+J_iD_{ii}-D_{ii}J_i$.  
\end{remark}


\subsection{Convergence analysis}
In this subsection, we conduct convergence analysis for the power method. Prior to this, we present the convergence conditions of the power method. For convenience,  we define the multiplicity of the eigenvalue of a DQM.
\begin{definition}[Multiplicity]\label{def:mul}
    Suppose that $\hat{A}\in\mathbb{DQ}^{n\times n}$ is a DQM, $\hat{\lambda}$ is an eigenvalue of $\hat{A}$. The multiplicity of  $\hat{\lambda}$, denoted by $\mathtt{m}(\hat{\lambda}, \hat{A})$, is defined as the cardinality of a maximal set of appreciably linearly independent eigenvectors in the eigenspace $\{\hat{\vx} \in \mathbb{DQ}^{n} \mid \hat{A}\hat{\vx} = \hat{\vx}\hat{\lambda}\}$.
\end{definition}

It follows from Definition~\ref{def:mul} that the multiplicity of $\lambda=\lambda_{s}+\lambda_d\epsilon$  is less than or equal to the algebraic multiplicity of  $\lambda_{s}$, i.e., 
\[\mathtt{m}(\hat{\lambda},\hat{A})\leq \mathtt{m}_g(\lambda_s,A_s)\leq \mathtt{m}_a(\lambda_s,A_s).\]

We now turn to the dominant eigenvalue of non-Hermitian DQMs. Suppose that $\hat{A}=A_s+A_d\epsilon\in\mathbb{DQ}^{n\times n}$. We say that an eigenvalue $\hat{\lambda}_1=\lambda_{1s}+\lambda_{1d}\epsilon$ is a dominant eigenvalue of $\hat{A}$,  if for any other eigenvalue $\hat{\mu}=\mu_{s}+\mu_{d}\epsilon$ of $\hat{A}$,  the following holds: 
\begin{equation*}
	|\lambda_{1s}|\geq |\mu_{s}|. 
\end{equation*}
The  corresponding eigenvector is called a dominant eigenvector. 
Furthermore, we say  $\hat{\lambda}_1=\lambda_{1s}+\lambda_{1d}\epsilon$ is a strict dominant eigenvalue of $\hat{A}$, with multiplicity $n_1$, if for any other eigenvalue $\hat{\mu}=\mu_{s}+\mu_{d}\epsilon$ of $\hat{A}$,  
\begin{equation*}
 |\lambda_{1s}|> |\mu_{s}|  \text{ and }  \mathtt{m}(\hat{\lambda}_1,\hat{A})=n_1. 
\end{equation*}




In the following, we propose the following assumption to ensure the convergence of the power method. This assumption generalizes the notion of the strict dominant eigenvalue for Hermitian DQMs, as discussed in \cite{cui2024power}, thereby encompassing it as a special case. 
\begin{assumption}\label{assump2}  
Let 
$\hat{A}=A_s+A_d\epsilon\in\mathbb{DQ}^{n\times n}$ be a DQM, and let $\lambda_{1s},\dots,\lambda_{ms}\in\mathbb{C}$ be distinct standard right eigenvalues of $A_s$. We assume that $|\lambda_{1s}|$  is larger than the magnitude of any other eigenvalues of $A_s$, 
and one of the following conditions holds: 
\begin{itemize} 
   \item[(i)] $\lambda_{1s}\in \mathbb R$, and there exists an eigenvalue $\hat{\lambda}_1=\lambda_{1s}+\lambda_{1d}\epsilon\in\mathbb{DR}$  such that $\mathtt{m}(\hat{\lambda}_1,\hat{A})=\mathtt{m}_a(\lambda_{1s},A_s)$;   
    \item[(ii)] there exists an eigenvalue $\hat{\lambda}_1=\lambda_{1s}+\lambda_{1d}\epsilon\in\mathbb{DC}$  such that $\mathtt{m}(\hat \lambda_{1},\hat A)=\mathtt{m}_g(\lambda_{1s},A_s)=\mathtt{m}_a(\lambda_{1s},A_s)=1$. 
\end{itemize}
\end{assumption}
We further present the Jordan-like form of DQMs that satisfy Assumptions~\ref{assump} and \ref{assump2}. Firstly, when $\hat{A}$ satisfies Assumption~\ref{assump} and $\lambda_{1s}\in\mathbb{R}$, the Jordan-like form given in \eqref{eq:diag} can be simplified to a more concise form, namely, $H_1$ can be reduced to a Jordan matrix.


\begin{lemma}\label{Lem:Jordan:mis}
     Suppose $\hat{A}\in\mathbb{DQ}^{n\times n}$ satisfies Assumption~\ref{assump}, and $\lambda_{1s}\in\mathbb{R}$. Then there exists  $\hat{P} = P_s + P_d\epsilon \in \mathbb{DQ}^{n \times n}$ and $H_i \in \mathbb{Q}^{n_i \times n_i}$ such that $\hat{P}^{-1}\hat{A}\hat{P}$ takes the following block diagonal Jordan-like form: 
    \begin{equation}\label{eq:diag:mis}
		\hat{P}^{-1}\hat{A}\hat{P}=
		\begin{pmatrix}
			\lambda_{1s}I_{n_1}+J_{1d}\epsilon&&&\\
			&J_2+H_2\epsilon&&\\
			&&\ddots&\\
			&&&J_m+H_m\epsilon
		\end{pmatrix}，
	\end{equation}
    where $J_{1d}$ is a Jordan matrix. If $\lambda_{1s}+\lambda_{1d}\epsilon$ is an eigenvalue of $\hat{A}$, then $\mathtt{m}(\lambda_{1s}+\lambda_{1d}\epsilon,\hat{A})=\mathtt{m}(\lambda_{1d},J_{1d})$.
\end{lemma}
\begin{proof}
    By Theorem \ref{Thm:Jordan like}, there exists a DQM $\hat{S}$ such that $\hat{S}^{-1}\hat{A}\hat{S}=$diag$(\lambda_{1s}I_{n_1}+H_1\epsilon,J_2+H_2\epsilon,\dots,J_m+H_m\epsilon)$ in the form of  \eqref{eq:diag}. By Theorem~ \ref{Thm:Jordan_Q}, we can choose $T\in\mathbb{Q}^{n_1\times n_1}$ such that $T^{-1}H_1T$ takes the Jordan canonical form $J_{1d}$. Let $\hat{P}=\hat{S}$diag$(T,I_{n-n_1})$, then $\hat{P}^{-1}\hat{A}\hat{P}$ takes the form of \eqref{eq:diag:mis}.
\end{proof}

Based on this lemma and Theorem~\ref{Thm:Jordan like}, we derive the Jordan-like form under Assumption~\ref{assump2}.
\begin{corollary}[Jordan-like form under Assumption~\ref{assump2}]\label{Cor:Jordan_2}
     Suppose $\hat{A}\in\mathbb{DQ}^{n\times n}$ satisfies Assumption~\ref{assump2}.  Then there exists a DQM $\hat{P} = P_s + P_d\epsilon \in \mathbb{DQ}^{n \times n}$ and matrices $H_i \in \mathbb{Q}^{n_i \times n_i}$ such that $\hat{P}^{-1}\hat{A}\hat{P}$ takes the following block diagonal Jordan-like form: 
    \begin{equation}\label{eq:diag_2}
		\hat{P}^{-1}\hat{A}\hat{P}=
		\begin{pmatrix}
			\hat{\lambda}_{1}I_{n_1}&&&\\
			&J_2+H_2\epsilon&&\\
			&&\ddots&\\
			&&&J_m+H_m\epsilon
		\end{pmatrix}.
	\end{equation}
\end{corollary}

Corollary~\ref{Cor:Jordan_2} plays a pivotal role in our convergence analysis, as it paves the way for a simplified expression of the matrix power $\hat{A}^k$.  Before presenting  the convergence analysis of Algorithm~\ref{assump}, we begin with the following lemma. 

\begin{lemma}
    Given two dual quaternion vectors $\hat{\vx}=\vx_s+\vx_d\epsilon\in\mathbb{DQ}^n$ and $\hat{\vy}=\vy_s+\vy_d\epsilon\in\mathbb{DQ}^n$. Suppose that both $\hat{\vx}$ and $\hat{\vx}+\hat{\vy}$ are appreciable, then
    \begin{equation}\label{eq:dual2}
        \Big|\textup{du}(\Vert \hat{\vx}+\hat{\vy}\Vert_2-\Vert\hat{\vx}\Vert_2) \Big|\leq \Vert\vy_d\Vert_2+\frac{2\Vert\vx_d\Vert_2}{\Vert\vx_s+\vy_s\Vert_2}\Vert\vy_s\Vert_2.
    \end{equation} 
\end{lemma}
\begin{proof}
    Through direct computations, we obtain the equality:
    \begin{equation*}
		\begin{aligned}
                &\textup{du}(\Vert \hat{\vx}+\hat{\vy}\Vert_2-\Vert\hat{\vx}\Vert_2)\\
                =&
                \frac{(\vx_s+\vy_s)^*(\vx_d+\vy_d)+(\vx_d+\vy_d)^*(\vx_s+\vy_s)}{2\Vert\vx_s+\vy_s\Vert_2}-\frac{\vx_s^*\vx_d+\vx_d^*\vx_s}{2\Vert\vx_s\Vert_2}\\  =&\left(\vx_s^*\vx_d+\vx_d^*\vx_s\right)\left(\frac{1}{2\Vert\vx_s+\vy_s\Vert_2}-\frac{1}{2\Vert\vx_s\Vert_2}\right)	+
                \frac{(\vx_s+\vy_s)^*\vy_d+\vy_d^*(\vx_s+\vy_s)+\vy_s^*\vx_d+\vx_d^*\vy_s}{2\Vert\vx_s+\vy_s\Vert_2}. 
		\end{aligned}
	\end{equation*}
    This combined with \eqref{eq:dual} and 
    $\Vert\vx_s+\vy_s\Vert_2\le \Vert\vx_s\Vert_2+\Vert\vy_s\Vert_2$ yields the following estimate:
    \begin{equation*}
		\begin{aligned}
                \Big|\textup{du}(\Vert \hat{\vx}+\hat{\vy}\Vert_2-\Vert\hat{\vx}\Vert_2)\Big|
                \leq& \Vert\vx_d\Vert_2\frac{\Vert\vx_s+\vy_s\Vert_2- \Vert\vx_s\Vert_2}{\Vert\vx_s+\vy_s\Vert_2}+\Vert\vy_d\Vert_2+\frac{\Vert\vx_d\Vert_2\Vert\vy_s\Vert_2}{\Vert\vx_s+\vy_s\Vert_2}
                \\
                \leq& \frac{\Vert\vx_d\Vert_2\Vert\vy_s\Vert_2}{\Vert\vx_s+\vy_s\Vert_2}+\Vert\vy_d\Vert_2+\frac{\Vert\vx_d\Vert_2\Vert\vy_s\Vert_2}{\Vert\vx_s+\vy_s\Vert_2}
                \\=&\Vert\vy_d\Vert_2+\frac{2\Vert\vx_d\Vert_2}{\Vert\vx_s+\vy_s\Vert_2}\Vert\vy_s\Vert_2.
		\end{aligned}
	\end{equation*}  
This completes the proof.
\end{proof}

Next we establish the following technical lemma.


\begin{lemma}\label{Lem:con_bd}
 Suppose that $\hat{A}\in\mathbb{DQ}^{n\times n}$  satisfies Assumption~\ref{assump2},  and let $\hat{P}$ defined in Corollary~\ref{Cor:Jordan_2} be partitioned as  $\hat{P}=(\hat{U}_1\ \hat{U}_2\ \cdots\ \hat{U}_m)$, where the columns of  $\hat{U}_1\in\mathbb{DQ}^{n\times n_1}$ are eigenvectors of $\hat{A}$ corresponding to the strictly dominant eigenvalue $\hat{\lambda}_1$,and $\hat{U}_i\in\mathbb{DQ}^{n\times n_i}$ for  $i = 2, \dots, m$.   
	Given an initial dual quaternion vector  $\hat{\vv}^{(0)}\in\mathbb{DQ}^{n\times 1}$ with unique decomposition: $\hat{\vv}^{(0)}=\sum_{j=1}^{m}\hat{U}_j\hat{\veta}_j$, where $\hat{\veta}_j\in\mathbb{DQ}^{n_j\times 1}$ for $j=1,\dots,m$. Suppose that $\hat{\veta}_1\in\mathbb{DQ}^{n_1}$ is appreciable. 
   Then for any $j=2,\dots,m$, the following estimates hold when $k\geq 2n_{j1}$: 
	\begin{equation}\label{eq:proof5}
		\begin{cases}
			\Vert\text{st}(\hat{U}_j(J_j+H_j\epsilon)^k\hat{\veta}_j)\Vert_2=O(k^{n_{j1}-1}|\lambda_{js}|^k)=\tilde{O}(|\lambda_{js}|^k),\\
			\Vert\text{du}(\hat{U}_j(J_j+H_j\epsilon)^k\hat{\veta}_j)\Vert_2=O(k(\frac{k-1}{2})^{2(n_{j1}-1)}|\lambda_{js}|^k)=\tilde{O}(|\lambda_{js}|^k). 
		\end{cases}
	\end{equation}
	Consequently, it holds that $\|\hat{U}_j(J_j+H_j\epsilon)^k\hat{\veta}_j\|_2=\tilde{O}_D(|\lambda_{js}|^k)$ and  $\Vert \hat{U}_j(J_j+H_j\epsilon)^k\hat{\veta}_j\Vert_{2^R}=\tilde{O}(|\lambda_{js}|^k)$ for $j=2,\dots,m$.  The notations $\tilde{O}(\cdot),\tilde{O}_D(\cdot)$ are defined by Definition \ref{def:O}.
\end{lemma}

\begin{proof} 
    We analyze two cases based on the value of $\lambda_{js}$: $\lambda_{js}\neq 0$ and $\lambda_{js}=0$.   
    
    {\bf Case (i) $\lambda_{js}\neq 0$.}  
    Let $n_{j1}$ denote the maximum size among all Jordan blocks associated with the eigenvalue  $\lambda_{js}$.
    It follows from the Jordan block structure that 
	\begin{equation}\label{eq:proof0.5}
		\begin{aligned}
			\Vert J_j^k\Vert_2\leq &\sqrt{\Vert J_j^k\Vert_1\Vert J_j^k\Vert_{\infty}}=\Vert J_j^k\Vert_{\infty}=\sum_{i=0}^{\min\{n_{j1}-1,k\}}
			\begin{pmatrix}
				k\\
				i
			\end{pmatrix}
			|\lambda_{js}|^{k-i}\\
			\leq& k^{n_{j1}-1}\sum_{i=0}^{\min\{n_{j1}-1,k\}}|\lambda_{js}|^{k-i}\leq c_jk^{n_{j1}-1}|\lambda_{js}|^k,
		\end{aligned}
	\end{equation}
	where 
	\begin{equation*}
		c_j=
		\begin{dcases}
			n_{j1},&\text{ if } |\lambda_{js}|=1,\\
			\frac{1-|\lambda_{js}|^{1-n_{j1}}}{1-|\lambda_{js}|^{-1}}, &\text{ if } |\lambda_{js}|\neq 1\text{ and }|\lambda_{js}|\neq 0.
		\end{dcases}
	\end{equation*}
     For any $k>0$, it holds that $(J_j+H_j\epsilon)^k=J_j^k+\sum_{i=0}^{k-1}J_j^iH_jJ_j^{k-1-i}\epsilon.$ 
	Thus the standard part satisfies 
	\begin{equation}\label{eq:proof1}
		\begin{aligned}
			&\Vert \text{st}(\hat{U}_j(J_j+H_j\epsilon)^k\hat{\veta}_j)\Vert_2=\Vert U_{js}J_j^k\veta_{js}\Vert_2\\
			\leq &\Vert U_{js}\Vert_2 \Vert J_j^k\Vert_2\Vert\veta_{js}\Vert_2\leq c_jk^{n_{j1}-1}\Vert U_{js}\Vert_2\Vert\veta_{js}\Vert_2|\lambda_{js}|^k. 
		\end{aligned}
	\end{equation}
	Similarly,  when $k\geq 2$, we may derive the following results for the dual part, 
	\begin{equation*}  
		\begin{split}
			&\Vert\text{du}(\hat{U}_j(J_j+H_j\epsilon)^k\hat{\veta}_j)\Vert_2\\ 
			=& \Vert U_{jd}J_j^k\veta_{js}+U_{js}J_j^k\veta_{jd}+U_{js}\sum_{i=0}^{k-1}J_j^iH_jJ_j^{k-1-i}\veta_{js}\Vert_2\\
			\leq&
			\Vert U_{jd}J_j^k\veta_{js}\Vert_2+\Vert U_{js}J_j^k\veta_{jd}\Vert_2+\Vert U_{js}\sum_{i=0}^{k-1}J_j^iH_jJ_j^{k-1-i}\veta_{js}\Vert_2\\
			\overset{\eqref{eq:proof0.5}}\leq&
			c_jk^{n_{j1}-1}\left(\Vert U_{jd}\Vert_2\Vert\veta_{js}\Vert_2+\Vert U_{js}\Vert_2\Vert\veta_{jd}\Vert_2\right)|\lambda_{js}|^k+\Vert U_{js}\Vert_2\Vert H_j\Vert_2\Vert\veta_{js}\Vert_2\sum_{i=0}^{k-1}\Vert J_j^i\Vert_2\Vert J_j^{k-1-i}\Vert_2.  
		\end{split}
        \raisetag{4.7ex}
	\end{equation*}
    Let $D_1=\Vert U_{jd}\Vert_2\Vert\veta_{js}\Vert_2+\Vert U_{js}\Vert_2\Vert\veta_{jd}\Vert_2, D_2=\Vert U_{js}\Vert_2\Vert H_j\Vert_2\Vert\veta_{js}\Vert_2$. Then we can derive that 
    	\begin{equation}\label{eq:proof2}
		\begin{split}
			&\Vert\text{du}(\hat{U}_j(J_j+H_j\epsilon)^k\hat{\veta}_j)\Vert_2\\  \le&c_jk^{n_{j1}-1}D_1|\lambda_{js}|^k+D_2\Big(2\Vert J_j^{k-1}\Vert_2+\sum_{i=1}^{k-2}\Vert J_j^i\Vert_2\Vert J_j^{k-1-i}\Vert_2\Big)\\
			\overset{\eqref{eq:proof0.5}}{\leq}&c_jk^{n_{j1}-1}D_1|\lambda_{js}|^k+D_2\Big(2c_j(k-1)^{n_{j1}-1}|\lambda_{js}|^{k-1}+\sum_{i=1}^{k-2}c_j^2\big(i(k-1-i)\big)^{n_{j1}-1}|\lambda_{js}|^{k-1}\Big)\\
            \leq&c_jk^{n_{j1}-1}D_1|\lambda_{js}|^k+\Big(2c_j(k-1)^{n_{j1}-1}+c_j^2(k-2)\left(\frac{k-1}{2}\right)^{2(n_{j1}-1)}\Big)D_2|\lambda_{js}|^{k-1}.
		\end{split}
        \raisetag{4.7ex}
	\end{equation}
It follows from equations \eqref{eq:proof1}--\eqref{eq:proof2} that equation \eqref{eq:proof5} holds for any $\lambda_{js}\neq 0$.

	{\bf Case (ii) $\lambda_{js}=0$.}  Since $J_j^{n_{j1}}=0$, for $k\geq 2n_{j1}$, we have 
	\begin{equation*}  
		\Vert\text{st}(\hat{U}_j(J_j+H_j\epsilon)^k\hat{\veta}_j)\Vert_2=\Vert U_{js}J_j^k\veta_{js})\Vert_2=0,
	\end{equation*}
	and 
	\begin{equation*}
		\begin{aligned}
			\Vert\text{du}(\hat{U}_j(J_j+H_j\epsilon)^k\hat{\veta}_j)\Vert_2=\Vert(U_{jd}J_j^k\veta_{js}+U_{js}J_j^k\veta_{jd}+U_{js}\sum_{i=0}^{k-1}J_j^iH_jJ_j^{k-1-i}\veta_{js})\Vert_2=
			0.
		\end{aligned}
	\end{equation*}
	Combining these two cases, we conclude that \eqref{eq:proof5} holds for all $j=2,\dots,m$.
    This completes the proof.
\end{proof}

We are now  ready to show that the sequence  $\hat{\vv}^{(k)}$ generated by Algorithm~\ref{alg_PM} converges to the eigenvector corresponding to the strict dominant eigenvalue linearly. 
 
\begin{theorem}\label{Thm:converge}
	 Under the same assumptions as in Lemma~\ref{Lem:con_bd}, the sequence $\hat{\vv}^{(k)}$ and $\hat{\lambda}^{(k)}$ generated by Algorithm~\ref{alg_PM} satisfy
	\begin{equation*} 
        \hat{\vv}^{(k)}
        =\frac{\hat{U}_1\hat{\lambda}_1^k\hat{\veta}_1}{\Vert \hat{U}_1\hat{\lambda}_1^k\hat{\veta}_1\Vert_2}\left(1+\tilde{O}_D\left(\left|\frac{\lambda_{2s}}{\lambda_{1s}}\right|^k\right)\right),
	\end{equation*} 
	and
	\begin{equation*}
		\hat{\lambda}^{(k)}=
        \begin{cases}
            \hat{\lambda}_1(1+\tilde{O}_D\left(\left|\frac{\lambda_{2s}}{\lambda_{1s}}\right|^k\right),&\text{if $\hat{A}$ satisfies Assumption~\ref{assump2}~(i),}       \\
            \frac{\hat{\veta}_1^*\hat{\lambda}_1\hat{\veta}_1}{\hat{\veta}_1^*\hat{\veta}_1}\left(1+\tilde{O}_D\left(\left|\frac{\lambda_{2s}}{\lambda_{1s}}\right|^k\right)\right),&\text{if $\hat{A}$ satisfies Assumption~\ref{assump2}~(ii).}
        \end{cases}
	\end{equation*} 
\end{theorem}
\begin{proof}
	By direct calculation, we obtain that
    \begin{equation*}
    \begin{aligned}
        \hat{A}^k\hat{\vv}^{(0)}& 
        =\hat{A}^k\hat{P}\begin{pmatrix}
            \hat{\veta}_1\\
            \hat{\veta}_2\\
            \vdots\\
            \hat{\veta}_m
        \end{pmatrix}=\hat{P}\begin{pmatrix}
            \hat{\lambda}_1^kI_{n_1}\\
            &(J_2+H_2\epsilon)^k\\
            &&\ddots\\
            &&&(J_m+H_m\epsilon)^k
        \end{pmatrix}\begin{pmatrix}
            \hat{\veta}_1\\
            \hat{\veta}_2\\
            \vdots\\
            \hat{\veta}_m
        \end{pmatrix}\\
        &=
        \begin{pmatrix}
            \hat{U}_1&\hat{U}_2&\cdots \hat{U}_m
        \end{pmatrix}\begin{pmatrix}
            \hat{\lambda}_1^kI_{n_1}\\
            &(J_2+H_2\epsilon)^k\\
            &&\ddots\\
            &&&(J_m+H_m\epsilon)^k
        \end{pmatrix}\begin{pmatrix}
            \hat{\veta}_1\\
            \hat{\veta}_2\\
            \vdots\\
            \hat{\veta}_m
        \end{pmatrix}\\
        &=\hat{U}_1\hat{\lambda}_1^k\hat{\veta}_1+\sum_{j=2}^m\hat{U}_j(J_j+H_j\epsilon)^k\hat{\veta}_j.
    \end{aligned}
    \end{equation*}
    Thus $\hat{\vv}^k$ can be expressed as  
	\begin{equation}\label{eq:proof0}
		\hat{\vv}^{(k)}=\frac{\hat{A}^k\hat{\vv}^{(0)}}{\left\Vert \hat{A}^k\hat{\vv}^{(0)}\right\Vert_2}=\frac{\hat{U}_1\hat{\lambda}_1^k\hat{\veta}_1+\sum_{j=2}^m\hat{U}_j(J_j+H_j\epsilon)^k\hat{\veta}_j}{\left\Vert \hat{U}_1\hat{\lambda}_1^k\hat{\veta}_1+\sum_{j=2}^m\hat{U}_j(J_j+H_j\epsilon)^k\hat{\veta}_j\right\Vert_2},
	\end{equation}
	  Next, we analyze the asymptotic behavior of each term of $\hat{\vv}^{(k)}$ in \eqref{eq:proof0}. By direct computation, we have
    \begin{equation}\label{eq:proof6}
       \left\Vert\text{st}(\hat{U}_1\hat{\lambda}_1^k\hat{\veta}_1)\right\Vert_2=O(|\lambda_{1s}|^k) \  \text{and}  \ 
			\left\Vert\text{du}(\hat{U}_1\hat{\lambda}_1^k\hat{\veta}_1)\right\Vert_2=O(k|\lambda_{1s}|^k)=\tilde{O}(|\lambda_{1s}|^k). 
    \end{equation}
    Combining the triangle inequality with \eqref{eq:dual2},\eqref{eq:proof5} and \eqref{eq:proof6}, we obtain the estimate for the denominator in \eqref{eq:proof0}:
	\begin{equation*}
		\left\Vert \hat{U}_1\hat{\lambda}_1^k\hat{\veta}_1+\sum_{j=2}^m\hat{U}_j(J_j+H_j\epsilon)^k\hat{\veta}_j\right\Vert_2=\left\Vert \hat{U}_1\hat{\lambda}_1^k\hat{\veta}_1\right\Vert_2+\tilde{O}_D(|\lambda_{2s}|^k).
	\end{equation*}
	Consequently, it holds that
	\begin{equation}\label{eq:proof7}
		\begin{aligned}
			\frac{\hat{\lambda}_1^k\hat{\veta}_1}{\Vert \hat{U}_1\hat{\lambda}_1^k\hat{\veta}_1+\sum_{j=2}^m\hat{U}_j(J_j+H_j\epsilon)^k\hat{\veta}_j\Vert_2}&=\frac{\hat{\lambda}_1^k\hat{\veta}_1}{\Vert \hat{U}_1\hat{\lambda}_1^k\hat{\veta}_1\Vert_2+\tilde{O}_D(|\lambda_{2s}|^k)}\\
			&=\frac{\hat{\lambda}_1^k\hat{\veta}_1}{\Vert \hat{U}_1\hat{\lambda}_1^k\hat{\veta}_1\Vert_2}+\tilde{O}_D\left(\left|\frac{\lambda_{2s}}{\lambda_{1s}}\right|^k\right),
		\end{aligned}
	\end{equation}
	and 
	\begin{equation}\label{eq:proof8}
		\begin{aligned}
			\frac{\Vert \hat{U}_i(J_i+H_i\epsilon)^k\hat{\veta}_i\Vert_{2^R}}{\Vert \hat{U}_1\hat{\lambda}_1^k\hat{\veta}_1+\sum_{j=2}^m\hat{U}_j(J_j+H_j\epsilon)^k\hat{\veta}_j\Vert_2}&=\frac{\tilde{O}(|\lambda_{2s}|^k)}{\Vert \hat{U}_1\hat{\lambda}_1^k\hat{\alpha}_1\Vert_2+\tilde{O}_D(|\lambda_{2s}|^k)}\\
			&=\tilde{O}_D\left(\left|\frac{\lambda_{2s}}{\lambda_{1s}}\right|^k\right),
		\end{aligned}
	\end{equation}
    for $i=2\dots,m$. 
	By substituting \eqref{eq:proof7} and \eqref{eq:proof8} into \eqref{eq:proof0}, we conclude:
	\begin{equation*}
		\hat{\vv}^{(k)}=\frac{\hat{U}_1\hat{\lambda}_1^k\hat{\veta}_1+\sum_{j=2}^m\hat{U}_j(J_j+H_j\epsilon)^k\hat{\veta}_j}{\Vert \hat{U}_1\hat{\lambda}_1^k\hat{\veta}_1+\sum_{j=2}^m\hat{U}_j(J_j+H_j\epsilon)^k\hat{\veta}_j\Vert_2}=\frac{\hat{U}_1\hat{\lambda}_1^k\hat{\veta}_1}{\Vert \hat{U}_1\hat{\lambda}_1^k\hat{\veta}_1\Vert_2}\left(1+\tilde{O}_D\left(\left|\frac{\lambda_{2s}}{\lambda_{1s}}\right|^k\right)\right),
	\end{equation*}
	\begin{equation*}
		\vy^{(k+1)}=\hat{A}\hat{\vv}^{(k)}=\frac{\hat{U}_1\hat{\lambda}_1^{k+1}\hat{\veta}_1}{\Vert \hat{U}_1\hat{\lambda}_1^k\hat{\veta}_1\Vert_2}\left(1+\tilde{O}_D\left(\left|\frac{\lambda_{2s}}{\lambda_{1s}}\right|^k\right)\right).
	\end{equation*} 
    For the eigenvalue approximation $\hat{\lambda}^{(k)}$ generated by Algorithm~\ref{alg_PM}, direct computation yields that
	\begin{equation*}
    \begin{aligned}
        \hat{\lambda}^{(k)}=(\hat{\vv}^{(k)})^*\hat{\vy}^{(k+1)}
        &=
        \frac{\hat{\veta}_1^*(\hat{\lambda}_1^*)^k\hat{U}_1^*\hat{U}_1\hat{\lambda}_1^{k+1}\hat{\veta}_1}{\hat{\veta}_1^*(\hat{\lambda}_1^*)^k\hat{U}_1^*\hat{U}_1\hat{\lambda}_1^{k}\hat{\veta}_1}\left(1+\tilde{O}_D\left(\left|\frac{\lambda_{2s}}{\lambda_{1s}}\right|^k\right)\right)\\
        &=
        \begin{cases}
            \hat{\lambda}_1\left(1+\tilde{O}_D\left(\left|\frac{\lambda_{2s}}{\lambda_{1s}}\right|^k\right)\right),&\text{if $\hat{A}$ satisfies Assumption~\ref{assump2}~(i),}       \\
            \frac{\hat{\veta}_1^*\hat{\lambda}_1\hat{\veta}_1}{\hat{\veta}_1^*\hat{\veta}_1}\left(1+\tilde{O}_D\left(\left|\frac{\lambda_{2s}}{\lambda_{1s}}\right|^k\right)\right),&\text{if $\hat{A}$ satisfies Assumption~\ref{assump2}~(ii).}
        \end{cases}
    \end{aligned}
	\end{equation*} 
	This completes the proof.
\end{proof}
 
        
As established in Theorem~\ref{Thm:converge}, the sequence generated by the power method converges linearly to the strictly dominant eigenvalue of a given non-Hermitian DQM and its associated eigenvector. The convergence rate of the eigenvalue is slower than that of the power method for Hermitian DQMs, whereas the latter exhibits a convergence rate of $\tilde{O}_D\left(\left|\frac{\lambda_{2s}}{\lambda_{1s}}\right|^{2k}\right)$.

From \eqref{eq:proof5} and \eqref{eq:proof8}, we can infer that if the standard part of the matrix contains large-order Jordan blocks, then the power method may have relatively large fluctuations before converging. This theoretical prediction is subsequently verified through numerical experiments.
 
\subsection{Assumption~\ref{assump2} is also necessary for the convergence}\label{Sec:necessary}


In this subsection, we establish that Assumption~\ref{assump2} is not only sufficient but also necessary for the convergence of the power method. Our analysis precedes an examination of five specific cases where the assumption is violated. 

\begin{lemma} \label{Lem:5cases}
Let $\hat{A}=A_s+A_d\epsilon\in\mathbb{DQ}^{n\times n}$ and let $\lambda_{1s},\dots,\lambda_{ms}\in\mathbb{C}$ 
be distinct 
standard eigenvalues of $A_s$, and $|\lambda_{1s}|\geq|\lambda_{2s}|\geq\cdots\geq|\lambda_{ms}|$. 
If Assumption~\ref{assump2} does not hold for $\hat{A}$, then it must fall into one of the following cases.
\begin{enumerate}[label=\textup{(\roman*)}]
    \item $|\lambda_{1s}|=|\lambda_{2s}|$. 
    \item $|\lambda_{1s}|>|\lambda_{2s}|$ and $\mathtt{m}_a(\lambda_{1s},A_s)>\mathtt{m}_g(\lambda_{1s},A_s)$.
    \item 
    $|\lambda_{1s}|>|\lambda_{2s}|$, $\mathtt{m}_a(\lambda_{1s},A_s)=\mathtt{m}_g(\lambda_{1s},A_s)>1$, and $\lambda_{1s}\in\mathbb{C}\setminus\mathbb{R}$.
    \item 
    $|\lambda_{1s}|>|\lambda_{2s}|$, $\mathtt{m}_a(\lambda_{1s},A_s)=\mathtt{m}_g(\lambda_{1s},A_s)$,  $\lambda_{1s}\in\mathbb{R}$,  
    and there exists an eigenvalue $\hat{\lambda}_1=\lambda_{1s}+\lambda_{1d}\epsilon$ of $\hat{A}$ such that $\mathtt{m}(\hat{\lambda}_1,\hat{A})<\mathtt{m}_a(\lambda_{1s},A_s)$, where $\lambda_{1d}\in\mathbb{C}$.
    \item 
    $|\lambda_{1s}|>|\lambda_{2s}|$, $\mathtt{m}_a(\lambda_{1s},A_s)=\mathtt{m}_g(\lambda_{1s},A_s)$,  $\lambda_{1s}\in\mathbb{R}$,  
    and there exists an eigenvalue $\hat{\lambda}_1=\lambda_{1s}+\lambda_{1d}\epsilon$ of $\hat{A}$ such that $\mathtt{m}(\hat{\lambda}_1,\hat{A})= \mathtt{m}_a(\lambda_{1s},A_s)>1$, where $\lambda_{1d}\in\mathbb{C}\setminus\mathbb{R}$. 

\end{enumerate}
\end{lemma}
\begin{proof}
    This lemma can be established by enumeration and  we omit the details for brevity. A comprehensive summary of all cases is provided in Table~\ref{Tab:5cases}. For convenience, $\mathtt{m}_a(\lambda_{1s},A_s)$,$\mathtt{m}_g(\lambda_{1s},A_s)$ and $\mathtt{m}(\hat{\lambda}_{1},\hat{A})$ are abbreviated as $\mathtt{m}_a$, $\mathtt{m}_g$, and $\mathtt{m}$ respectively in this table.
\end{proof}

\begin{table}[]
\centering
\begin{tblr}{
  colspec = {|c|c|c|c|c|}, 
  hlines, vline{2,3}={2,3}{red},vline{4,5}={3}{red},
  cell{2}{2} = {bg=pink}, 
  cell{3}{4} = {bg=pink}, 
}
\SetCell[r=3]{c} \makecell[c]{(i) $|\lambda_{1s}|=|\lambda_{2s}|$.} 
& \SetCell[c=4]{c} \makecell[c]{(ii) $|\lambda_{1s}|>|\lambda_{2s}|$, \\$\mathtt{m}_a>\mathtt{m}_g$.} 
& & & \\  \cline[red]{2}
& \SetCell[r=2]{c} {Assumption~\ref{assump2}\\ (ii) $|\lambda_{1s}|>|\lambda_{2s}|$,\\ $\mathtt{m}_a=\mathtt{m}_g=1$.} 
& \SetCell[r=2]{c} \makecell[c]{(iii) $|\lambda_{1s}|>|\lambda_{2s}|$,\\
$\mathtt{m}_a=\mathtt{m}_g>1$,\\
$\lambda_{1s}\notin\mathbb{R}$.} 
& \SetCell[c=2]{c} \makecell[c]{(iv) $|\lambda_{1s}|>|\lambda_{2s}|$,\\
$\mathtt{m}_a=\mathtt{m}_g>1$,\\
$\lambda_{1s}\in\mathbb{R}$,\\
$\mathtt{m}<\mathtt{m}_a$.} 
& \\ \cline[red]{4}
& 
& 
& \makecell[c]{Assumption~\ref{assump2}\\ (i) $|\lambda_{1s}|>|\lambda_{2s}|$,\\ $\mathtt{m}_a=\mathtt{m}_g>1$,\\
$\lambda_{1s}\in\mathbb{R},\lambda_{1d}\in\mathbb{R}$,\\
$\mathtt{m}=\mathtt{m}_a$.} 
& \makecell[c]{(v) $|\lambda_{1s}|>|\lambda_{2s}|$,\\
$\mathtt{m}_a=\mathtt{m}_g>1$,\\
$\lambda_{1s}\in\mathbb{R},\lambda_{1d}\notin\mathbb{R}$,\\
$\mathtt{m}_a=\mathtt{m}$.} \\ \cline[red]{2,4}
\end{tblr}
\caption{Categorization of cases in Lemma~\ref{Lem:5cases} and Assumption~\ref{assump2}.}
\label{Tab:5cases}
\end{table}

\begin{theorem}\label{Thm:divergence_quater}
    Let $\hat{A}=A_s+A_d\epsilon\in\mathbb{DQ}^{n\times n}$ and let $\lambda_{1s},\dots,\lambda_{ms}\in\mathbb{C}$  be distinct 
standard eigenvalues of $A_s$ that satisfy $|\lambda_{1s}|\geq|\lambda_{2s}|\geq\cdots\geq|\lambda_{ms}|$.  
If Assumption~\ref{assump2} is violated for $\hat{A}$, then the power method (Algorithm~\ref{alg_PM}) either diverges or fails to yield the desired eigenpair unless a sufficiently good initial vector is chosen. 
  
\end{theorem}
\begin{proof} 

If $\hat{A}$ satisfies \textup{(i)} or \textup{(ii)}, the power method loses its validity, as the standard part of $\hat{\vv}^{(k)}$ does not converge to the eigenvector of $A_s$, as reported in \cite{Mitchell_1967,doi:10.1137/0710035}.

If $\hat{A}$ satisfies \textup{(iii)}, by Theorem \ref{Thm:Jordan like}, there exists a DQM $\hat{P}$ such that $\hat{P}^{-1}\hat{A}\hat{P}=$diag$(\lambda_{1s}I_{n_1}+H_1\epsilon,J_2+H_2\epsilon,\dots,J_m+H_m\epsilon)$ in the form of  \eqref{eq:diag}. Denote the $i$-th column of $\hat{P}$ by $\vu_i=\vu_{is}+\vu_{id}\epsilon$ for any $i\in[n]$. For an initial vector $\hat{\vv}^{(0)}=\vv^{(0)}_s+\vv^{(0)}_d\epsilon$, $\vv^{(0)}_s$ has a unique decomposition $\vv^{(0)}_s=\sum_{i=1}^n\vu_{is}\alpha_{i}$, where $\alpha_{i}\in\mathbb{Q}$. Then the standard part of $\hat{\vv}^{(k)}$ produced by the power method converges to $\beta^{-1}\sum_{i=1}^{n_1}\vu_{is}\alpha_{i}\lambda_{1s}^k$, where $\beta=\Vert\sum_{i=1}^{n_1}\vu_{is}\alpha_{i}\lambda_{1s}^k\Vert_2$. Consider the residual:
\begin{equation*}
    A_s(\beta^{-1}\sum_{i=1}^{n_1}\vu_{is}\alpha_{i}\lambda_{1s}^k)-(\beta^{-1}\sum_{i=1}^{n_1}\vu_{is}\alpha_{i}\lambda_{1s}^k)\gamma^{-1}\lambda_{1s}\gamma=\beta^{-1}\sum_{i=1}^{n_1}\vu_{is}(\lambda_{1s}\alpha_{i}\lambda_{1s}^k\gamma^{-1}-\alpha_i\lambda_{1s}^k\gamma^{-1}\lambda_{1s})\gamma,
\end{equation*}
where $\gamma\in\mathbb{Q}$ is a non-zero quaternion. Thus, $\sum_{i=1}^{n_1}\vu_{is}\alpha_{i}\lambda_{1s}^k$ is an eigenvector of $A_s$ if and only if there exists a non-zero quaternion $\gamma$ such that $\lambda_{1s}\alpha_{i}\lambda_{1s}^k\gamma^{-1}=\alpha_i\lambda_{1s}^k\gamma^{-1}\lambda_{1s}$ for $i=1,\dots,n_1$. Given that $\lambda_{1s}\in\mathbb{C}\setminus\mathbb{R}$, this is equivalent to $\alpha_{i}\lambda_{1s}^k\gamma^{-1}\in\mathbb{C}$ for $i=1,\dots,n_1$. Upon denoting $\alpha_i=a_i+b_i\jj$ for $i=1,\dots,n_1$, where $a_i,b_i\in\mathbb{C}$, this reduces to the to the rank of the complex vector set $\{[a_1,b_1],\dots,[a_{n_1},b_{n_1}]\}$ being 1.   Consequently, $\hat{\vv}^{(k)}$ fails to converge to the target eigenvector if the rank of the complex vector set $\{[a_1,b_1],\dots,[a_{n_1},b_{n_1}]\}$ exceeds 1.

If $\hat{A}$ satisfies \textup{(iv)}, then for any eigenvalue $\hat{\xi}=\lambda_{1s}+\xi_{d}\epsilon$ with standard part $\lambda_{1s}$, we have $\mathtt{m}(\hat{\xi},\hat{A})<\mathtt{m}_g(\lambda_{1s},A_s)=\mathtt{m}_a(\lambda_{1s},A_s)$ by Lemma~\ref{Lem:Jordan:mis}.  Let $\hat{\vu}_1,\dots,\hat{\vu}_t$ denote the eigenvectors of $\hat{A}$ associated with the eigenvalue $\hat{\xi}$; these vectors are appreciably linearly independent and take the form $\hat{\vu}_i=\vu_{is}+\vu_{id}\epsilon$ for $i=1,\dots,t$, where $t=\mathtt{m}(\hat{\xi},\hat{A})$. There exists $\vu_{is}\in\mathbb{Q}^n$, $i=t+1,\dots,n$, and $P_s=(\vu_{1s}\ \vu_{2s}\ \cdots \vu_{ns})$ such that $P_s^{-1}A_sP_s=$diag$(\lambda_{1s}I_{n_1},J_2,\dots,J_m)$, where $J_i$ is the Jordan matrix corresponding to $\lambda_{ms}$. We can express the initial vector $\hat{\vv}^{(0)}$ as $\hat{\vv}^{(0)}=\sum_{i=1}^n\vu_{is}\alpha_{i}+\vw\epsilon$, with $\alpha_i\in\mathbb{Q}$, $\vw\in\mathbb{Q}^n$. Then the standard part of $\hat{\vv}^{(k)}$ produced by the power method converges to $\sum_{i=1}^{n_1}\vx_{is}\alpha_{i}$. 
It follows that $\hat{\vv}^{(k)}$ will not converge to the eigenvector associated with $\xi$ if $\alpha_i\neq 0$ for some $i=t+1,\dots,n_1$.

If $\hat{A}$ satisfies \textup{(v)}, by Lemma~\ref{Lem:Jordan:mis}, there exists a DQM $\hat{P}$ such that $\hat{P}^{-1}\hat{A}\hat{P}=$diag$(\hat{\lambda}_{1}I_{n_1},J_2+H_2\epsilon,\dots,J_m+H_m\epsilon)$ in the form of  \eqref{eq:diag:mis}. Denote the i-th column of $\hat{P}$ as $\hat{\vu}_i=\vu_{is}+\vu_{id}\epsilon$. For an initial vector $\hat{\vv}^{(0)}=\vv^{(0)}_s+\vv^{(0)}_d\epsilon$, $\hat{\vv}^{(0)}$ has a unique decomposition $\hat{\vv}^{(0)}=\sum_{i=1}^n\hat{\vu}_{i}\hat{\alpha}_{i}$, where $\hat{\alpha}_{i}=\alpha_{is}+\alpha_{id}\epsilon\in\mathbb{DQ}$. Then the standard part of $\hat{\vv}^{(k)}$ produced by the power method converges to $\hat{\beta}^{-1}\sum_{i=1}^{n_1}\hat{\vu}_{i}\hat{\alpha}_{i}\hat{\lambda}_{1}^k$, where $\hat{\beta}=\Vert\sum_{i=1}^{n_1}\hat{\vu}_{i}\hat{\alpha}_{i}\hat{\lambda}_{1}^k\Vert_2$. Consider the residual:
\begin{equation*}
    \hat{A}(\hat{\beta}^{-1}\sum_{i=1}^{n_1}\hat{\vu}_{i}\hat{\alpha}_{i}\hat{\lambda}_{1}^k)-(\hat{\beta}^{-1}\sum_{i=1}^{n_1}\hat{\vu}_{i}\hat{\alpha}_{i}\hat{\lambda}_{1}^k)\hat{\gamma}^{-1}\hat{\lambda}_{1}\hat{\gamma}=\hat{\beta}^{-1}\sum_{i=1}^{n_1}\hat{\vu}_{i}(\hat{\lambda}_{1}\hat{\alpha}_{i}\hat{\lambda}_{1}^k\hat{\gamma}^{-1}-\hat{\alpha}_i\hat{\lambda}_{1}^k\hat{\gamma}^{-1}\hat{\lambda}_{1})\hat{\gamma},
\end{equation*}
where $\hat{\gamma}=\gamma_s+\gamma_d\epsilon\in\mathbb{DQ}$ is an appreciable dual quaternion. Thus, $\sum_{i=1}^{n_1}\hat{\vu}_{i}\hat{\alpha}_{i}\hat{\lambda}_{1}^k$ is an eigenvector of $\hat{A}$ if and only if there exists an appreciable dual quaternion $\hat{\gamma}$ such that $\hat{\lambda}_{1}\hat{\alpha}_{i}\hat{\lambda}_{1}^k\hat{\gamma}^{-1}=\hat{\alpha}_i\hat{\lambda}_{1}^k\hat{\gamma}^{-1}\hat{\lambda}_{1}$ for $i=1,\dots,n_1$. Given that $\lambda_{1s}\in\mathbb{R}$ and $\lambda_{1d}\in\mathbb{C}\setminus\mathbb{R}$, this is equivalent to $\alpha_{is}\lambda_{1s}^k\gamma_s^{-1}\in\mathbb{C}$ for $i=1,\dots,n_1$. Upon denoting $\alpha_{is}=a_i+b_i\jj$ for $i=1,\dots,n_1$, where $a_i,b_i\in\mathbb{C}$, this reduces to the to the rank of the complex vector set $\{[a_1,b_1],\dots,[a_{n_1},b_{n_1}]\}$ being 1.   Consequently, $\hat{\vv}^{(k)}$ fails to converge to the target eigenvector if the rank of the complex vector set $\{[a_1,b_1],\dots,[a_{n_1},b_{n_1}]\}$ exceeds 1. 
\end{proof}

This theorem illustrates the necessity of Assumption~\ref{assump2}, as can also be observed from Examples~\ref{exp:fail(iii)},~\ref{exp:fail(iv)}, and \ref{exp:fail(v)}.

\begin{example}\label{exp:fail(iii)}
    Suppose that $\hat{A}=A_s+A_d\epsilon$, where
    \begin{equation*}
        A_s=\begin{pmatrix}
            \ii&0\\
            0&\ii
        \end{pmatrix} \text{ and } 
        A_d=\begin{pmatrix}
            \ii&0\\
            0&\ii
        \end{pmatrix}.
    \end{equation*}
    Then $\hat{A}$ satisfied Lemma~\ref{Lem:5cases}~(iii). Let $\hat{\vv}^{(0)}=[1+\epsilon,\jj]^\top$ be the initial vector of the power method. By direct computation, we obtain $\hat{\vv}^{(4k)}=[\frac{1}{\sqrt{2}}+\frac{1}{2\sqrt{2}}\epsilon,\frac{1}{\sqrt{2}}\jj-\frac{1}{2\sqrt{2}}\epsilon\jj]^\top$. However, since $[\frac{1}{\sqrt{2}},\frac{1}{\sqrt{2}}\jj]^\top$ is not an eigenvector of $\hat{A_s}$, the sequence $\hat{\vv}^{(k)}$ fails to converge to any eigenvector of $\hat{A}$. 
\end{example}

\begin{example}\label{exp:fail(iv)}
    Suppose that $\hat{A}=A_s+A_d\epsilon$, where
    \begin{equation*}
        A_s=\begin{pmatrix}
            1&0&0\\
            0&1&0\\
            0&0&1
        \end{pmatrix} \text{ and } 
        A_d=\begin{pmatrix}
            2&0&0\\
            0&1&1\\
            0&0&1
        \end{pmatrix}.
    \end{equation*}
    Then $\hat{A}$ satisfies Lemma~\ref{Lem:5cases}~(iv). Let $\hat{\vv}^{(0)}=[1,1,1]^\top$, we have $\hat{\vv}^{(k)}=\frac{1}{\sqrt{3}}[1-\frac{1}{3}k\epsilon, 1-\frac{1}{3}k, 1+\frac{2}{3}k\epsilon], \lambda^{(k)}=1+\frac{5}{3}\epsilon$. 
    Consequently, $\hat{\vv}^{(k)}$ diverges and $\lambda^{(k)}$ fails to converge to any eigenvalue of $\hat{A}$.
\end{example}

\begin{example}\label{exp:fail(v)}
    Let $\hat{A}=A_s+A_d\epsilon$, where
    \begin{equation*}
        A_s=\begin{pmatrix}
            1&0\\
            0&1
        \end{pmatrix} \text{ and } 
        A_d=\begin{pmatrix}
            \ii&0\\
            0&\ii
        \end{pmatrix}.
    \end{equation*}
    Then $\hat{A}$ satisfies Lemma~\ref{Lem:5cases}~(v). Let $\hat{\vv}^{(0)}=[1,\jj]^{\top}$, 
     we have $\hat{\vv}^{(k)}=\frac{1}{\sqrt{2}}[1+k\ii\epsilon,(1+k\ii\epsilon)\jj]^\top$ and $\hat{\lambda}^{(k)}=1$. 
    This indicates that $\hat{\vv}^{(k)}$ diverges and $\lambda^{(k)}$ fails to converge to any eigenvalue of $\hat{A}$.
\end{example}



\subsection{Non-Hermitian dual complex matrices}
Although non-Hermitian dual complex matrices are a special case of DQMs, they possess the property of commutativity. This can lead to unique characteristics. To account for this, we adapt Assumption~\ref{assump2} to Assumption~\ref{assump3} as follows.

\begin{assumption}\label{assump3}  
Let 
$\hat{A}=A_s+A_d\epsilon\in\mathbb{DC}^{n\times n}$ be a non-Hermitian dual complex matrix, and let $\lambda_{1s},\dots,\lambda_{ms}\in\mathbb{C}$ denote the eigenvalues of $A_s$. We assume that $|\lambda_{1s}|$  is larger than the magnitude of any other eigenvalues of $A_s$, and there exists an eigenvalue $\hat{\lambda}_1=\lambda_{1s}+\lambda_{1d}\epsilon\in\mathbb{DC}$ of $\hat{A}$ such that $\mathtt{m}(\hat{\lambda}_1,\hat{A})=\mathtt{m}_a(\lambda_{1s},A_{s})$.
\end{assumption}
Under Assumption~\ref{assump3}, we establish the convergence analysis for non-Hermitian dual complex matrices. The proof closely mirrors that for the non-Hermitian DQMs case and is therefore omitted here.
\begin{lemma}[Jordan-like form under Assumption~\ref{assump3}]\label{Lem:Jordan_3}
     Suppose $\hat{A}\in\mathbb{DC}^{n\times n}$ satisfies Assumption~\ref{assump3}.  Then there exists $\hat{P}=P_s+P_d\epsilon \in \mathbb{DC}^{n \times n}$ and  $H_i \in \mathbb{C}^{n_i \times n_i}$ such that $\hat{P}^{-1}\hat{A}\hat{P}$ takes the following block diagonal Jordan-like form: 
    \begin{equation}\label{eq:diag_3}
		\hat{P}^{-1}\hat{A}\hat{P}=
		\begin{pmatrix}
			\hat{\lambda}_{1}I_{n_1}&&&\\
			&J_2+H_2\epsilon&&\\
			&&\ddots&\\
			&&&J_m+H_m\epsilon
		\end{pmatrix}.
	\end{equation}
\end{lemma}

\begin{theorem}\label{Thm:con_bd_adj}
 Suppose that $\hat{A}\in\mathbb{DC}^{n\times n}$  satisfies Assumption~\ref{assump3},  and let  $\hat{P}$ defined in Lemma~\ref{Lem:Jordan_3} be partitioned as  $\hat{P}=(\hat{U}_1\ \hat{U}_2\ \cdots\ \hat{U}_m)$, where the columns of $\hat{U}_1\in\mathbb{DC}^{n\times n_1}$ are eigenvectors of $\hat{A}$ corresponding to the strictly dominant eigenvalue $\hat{\lambda}_1$, and $\hat{U}_i\in\mathbb{DC}^{n\times n_i}$ for any $i = 2, \dots, m$.   
	Given an initial  dual complex vector  $\hat{\vv}^{(0)}\in\mathbb{DC}^{n\times 1}$,  consider its unique decomposition: $\hat{\vv}^{(0)}=\sum_{j=1}^{m}\hat{U}_j\hat{\veta}_j$, where $\hat{\veta}_j\in\mathbb{DC}^{n_j\times 1}$ for all $j=1,\dots,m$ and $\hat{\veta}_1\in\mathbb{DC}^{n_1}$ is appreciable. 
   Then, the sequence $\hat{\vv}^{(k)}$ and $\hat{\lambda}^{(k)}$ generated by Algorithm~\ref{alg_PM} satisfies
	\begin{equation*} 
        \hat{\vv}^{(k)}
        =\frac{\hat{U}_1\hat{\veta}_1}{\Vert \hat{U}_1\hat{\veta}_1\Vert_2}\left(\frac{\hat{\lambda}_1}{|\hat{\lambda}_1|}\right)^k\left(1+\tilde{O}_D\left(\left|\frac{\lambda_{2s}}{\lambda_{1s}}\right|^k\right)\right),
	\end{equation*} 
	and
	\begin{equation*}
		\hat{\lambda}^{(k)}=
            \hat{\lambda}_1\left(1+\tilde{O}_D\left(\left|\frac{\lambda_{2s}}{\lambda_{1s}}\right|^k\right)\right).
	\end{equation*} 
\end{theorem}
Through an analysis similar to that in Section~\ref{Sec:necessary}, we can conclude that Assumption~\ref{assump3} is also necessary for Theorem~\ref{Thm:con_bd_adj}. 
\begin{theorem}\label{Thm:divergence_complex}
    Let $\hat{A}=A_s+A_d\epsilon\in\mathbb{DC}^{n\times n}$.  
If Assumption~\ref{assump3} is violated for $\hat{A}$, then the power method (Algorithm~\ref{alg_PM}) either diverges or fails to yield the desired eigenpair unless a sufficiently good initial vector is chosen. 
\end{theorem}

\subsection{Dual complex adjoint matrix based power method}
As established in Theorem~\ref{Thm:adjoint}, the eigenvalues of non-Hermitian DQMs can be computed via their dual complex adjoint matrices. This computation is outlined in Algorithm~\ref{alg_PM_adj}, which is known as DCAM-PM \cite{chen2025improved}. The algorithmic framework follows the same structure as the one developed for Hermitian DQMs \cite{chen2025improved}.

Algorithm~\ref{alg_PM_adj} transforms the eigenvalue problem of a non-Hermitian DQM into the eigenvalue problem of a non-Hermitian dual complex matrix via the dual complex adjoint matrix. 

\begin{algorithm}[t]
	\caption{Dual complex adjoint matrix based power method (DCAM-PM) for computing the dominant eigenvalue of  non-Hermitian DQMs.}\label{alg_PM_adj} 
	\begin{algorithmic}[1]
		\Require a non-Hermitian DQM $\hat{A}\in\mathbb{DQ}^{n\times n}$, an initial vector $\hat{\vv}^{(0)}\in\mathbb{DQ}^n$, the maximal iteration number $k_{\max}$ and the tolerance $\delta$
        \State Compute $\hat{B}=\mathcal{J}(\hat{A})$,$\hat{\vw}^{(0)}=\mathcal{F}(\hat{\vv}^{(0)})$.
		\For{$k=1,2,\dots,k_{\max}$}
		\State $\hat{\vy}^{(k)}=\hat{B}\hat{\vw}^{(k-1)}$,
		\State $\hat{\lambda}^{(k-1)}=(\hat{\vu}^{(k-1)})^*\hat{\vy}^{(k)}$.
		\If{$\Vert \hat{\vw}^{(k)}-\hat{\vv}^{(k-1)}\hat{\lambda}^{(k-1)}\Vert_{2^R}\leq \delta$}
        \State Break.
		\EndIf
		\State $\hat{\vw}^{(k)}=\frac{\hat{\vy}^{(k)}}{\Vert \hat{\vy}^{(k)}\Vert_2}$.
		\EndFor
        \State
        Compute $\hat{\vv}^{(k-1)}=\mathcal{F}^{-1}(\vw^{(k-1)})$,
		\Ensure $\hat{\vv}^{(k-1)}$, $\hat{\lambda}^{(k)}$.
	\end{algorithmic}
\end{algorithm}

Suppose $\hat{A}\in\mathbb{DQ}^{n\times n}$, for $\mathcal{J}(\hat{A})$ to satisfy Assumption~\ref{assump3}, it is necessary that $\hat{A}$ satisfies Assumption~\ref{assump2}(i). By Theorem~\ref{Thm:adjoint} and Theorem~\ref{Thm:con_bd_adj}, we can obtain the convergence results of Algorithm~\ref{alg_PM_adj}:
\begin{theorem}
    Suppose $\hat{A}\in\mathbb{DQ}^{n\times n}$ satisfies Assumption~\ref{assump2}(i), and let $\hat{P}$ defined in Corollary~\ref{Cor:Jordan_2} be partitioned as  $\hat{P}=(\hat{U}_1\ \hat{U}_2\ \cdots\ \hat{U}_m)$, where the columns of $\hat{U}_1\in \mathbb{DQ}^{n\times n_1}$ are eigenvectors of $\hat{A}$ corresponding to the strictly dominant eigenvalue $\hat{\lambda}_1$, and $\hat{U}_i\in\mathbb{DQ}^{n\times n_i}$ for any $i = 2, \dots, m$.   
	Given an initial dual quaternion vector  $\hat{\vv}^{(0)}\in\mathbb{DQ}^{n\times 1}$,  consider its unique decomposition: $\hat{\vv}^{(0)}=\sum_{j=1}^{m}\hat{U}_j\hat{\veta}_j$, where $\hat{\veta}_j\in\mathbb{DQ}^{n_j\times 1}$ for all $j=1,\dots,m$ and $\hat{\veta}_1\in\mathbb{DQ}^{n_1}$ is appreciable.
    Then, the sequence $\hat{\vv}^{(k)}$ and $\hat{\lambda}^{(k)}$ generated by Algorithm~\ref{alg_PM_adj} satisfies
	\begin{equation*} 
        \hat{\vv}^{(k)}
        =\frac{\hat{U}_1\hat{\veta}_1}{\Vert \hat{U}_1\hat{\veta}_1\Vert_2}\left(\frac{\hat{\lambda}_1}{|\hat{\lambda}_1|}\right)^k\left(1+\tilde{O}_D\left(\left|\frac{\lambda_{2s}}{\lambda_{1s}}\right|^k\right)\right),
	\end{equation*} 
	and
	\begin{equation*}
		\hat{\lambda}^{(k)}=
            \hat{\lambda}_1\left(1+\tilde{O}_D\left(\left|\frac{\lambda_{2s}}{\lambda_{1s}}\right|^k\right)\right).
	\end{equation*} 
\end{theorem}
\begin{theorem}\label{Thm:divergence_adjoint}
    Let $\hat{A}=A_s+A_d\epsilon\in\mathbb{DQ}^{n\times n}$.  
If Assumption~\ref{assump2}~(i) is violated for $\hat{A}$, then the dual complex adjoint matrix based power method (Algorithm~\ref{alg_PM_adj}) either diverges or fails to yield the desired eigenpair. 
  
\end{theorem}
As established in these two theorems, the convergence rate of Algorithm~\ref{alg_PM_adj} is comparable to that of Algorithm~\ref{alg_PM}, while its convergence conditions are more stringent than those of Algorithm~\ref{alg_PM}. 
Consequently, although PM and DCAM-Pm exhibit similar convergence behavior for Hermitian DQMs, DCAM-PM is subject to stricter convergence requirements than PM when applied to non-Hermitian DQMs.

\section{Numerical Experiments}
In this section, we present the numerical experiments for computing the eigenvalue of the non-Hermitian DQMs by the power method (Algorithm~\ref{alg_PM}) and the dual complex adjoint matrix based power method (Algorithm~\ref{alg_PM_adj}). The power method is referred to as PM, and the dual complex adjoint matrix based power method is abbreviated as DCAM-PM, respectively. The computations are performed on an Intel Core i9-14900HX @ 2.20GHz/32GB computer, using MATLAB as the programming environment.
Our code is available at \url{https://github.com/BUAA-HaoYang/DQ-toolbox}.

Unless otherwise specified, we set the maximum number of iterations to $k_{\max}=1000$ and the tolerance to $\delta=10^{-10}$. 
We report the iterative performance measured by the $2^R$-norm of the residual value, defined as
\begin{equation}
\textup{Res}=\Vert \hat{A}\hat{\vv}^{(k)}-\hat{\vv}^{(k)}\hat{\lambda}^{(k)}\Vert_{2^R}. 
\end{equation}

\subsection{Laplacian matrices for formation control} 
 In the multi-agent formation control, the Laplacian matrix of the unit dual quaternion directed graph (UDQDG) \cite{qi2025unit} plays a key role, especially for leader-follower structures \cite{Liu2025distributed}.
We begin by examining the numerical performance of PM and DCAM-PM applied to such Laplacian matrices.
Suppose we have a UDQDG  $\phi=(G,\mathbb{\hat{U}},\varphi)$, where $\mathbb{\hat{U}}$ is the set of unit dual quaternion, $G=(V,E)$ is a directed graph, $\varphi(i,j)=\hat{q}_{d_{ij}}\in\mathbb{\hat{U}}$ is the weight of the arc $(i,j)$ for $(i,j)\in E$. 
The Laplacian matrix of $\phi$ is defined by
\begin{equation*}
    \hat{L}=D-\hat{A},
\end{equation*}
where $D$ is a real diagonal matrix whose $i$th diagonal element is equal to the out-degree $d_i$ (the number of arcs going out from $i$) of the $i$th vertex, and $\hat{A}=(\hat{a}_{ij})$ with
\begin{equation*}
    \hat{a}_{ij}=\begin{cases}
        \hat{q}_{d_{ij}}, &\text{if } (i,j)\in E,\\
        0,&\text{otherwise}.
    \end{cases}
\end{equation*}

We first consider directed cycles demonstrated in \textup{Fig}~\ref{fig1:main}. The  corresponding Laplacian matrices for 3 and 4 vertices are 
    \begin{equation}\label{Lap:cycle} 
        \hat{L}_3^{c}=\begin{pmatrix}
            1&-\hat{q}_{d_{12}}&0\\
            0&1&-\hat{q}_{d_{23}}\\
            -\hat{q}_{d_{31}}&0&1
        \end{pmatrix},  \ \ 
        \hat{L}_4^{c}=\begin{pmatrix}
            1&-\hat{q}_{d_{12}}&0&0\\
            0&1&-\hat{q}_{d_{23}}&0\\
            0&0&1&-\hat{q_{d_{34}}}\\
            -\hat{q}_{d_{41}}&0&0&1\\
        \end{pmatrix},
    \end{equation}
    respectively. 
    Both $\hat{L}_3^{c}$ and $\hat{L}_4^{c}$ are non-Hermitian.

    \begin{figure}[htbp]
    \centering
    \begin{subfigure}{0.3\textwidth}
        \centering
        \includegraphics[width=\textwidth]{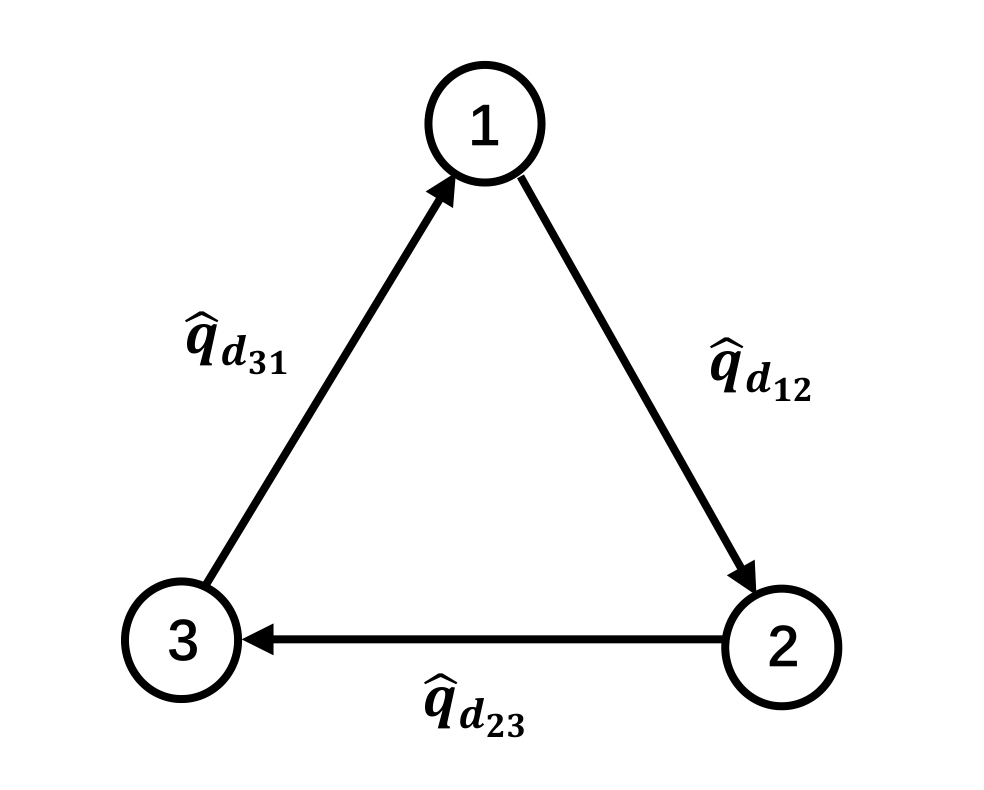}
        \caption{$|V|=3$}
        \label{fig1:sub1}
    \end{subfigure}
    \begin{subfigure}{0.3\textwidth}
        \centering
        \includegraphics[width=\textwidth]{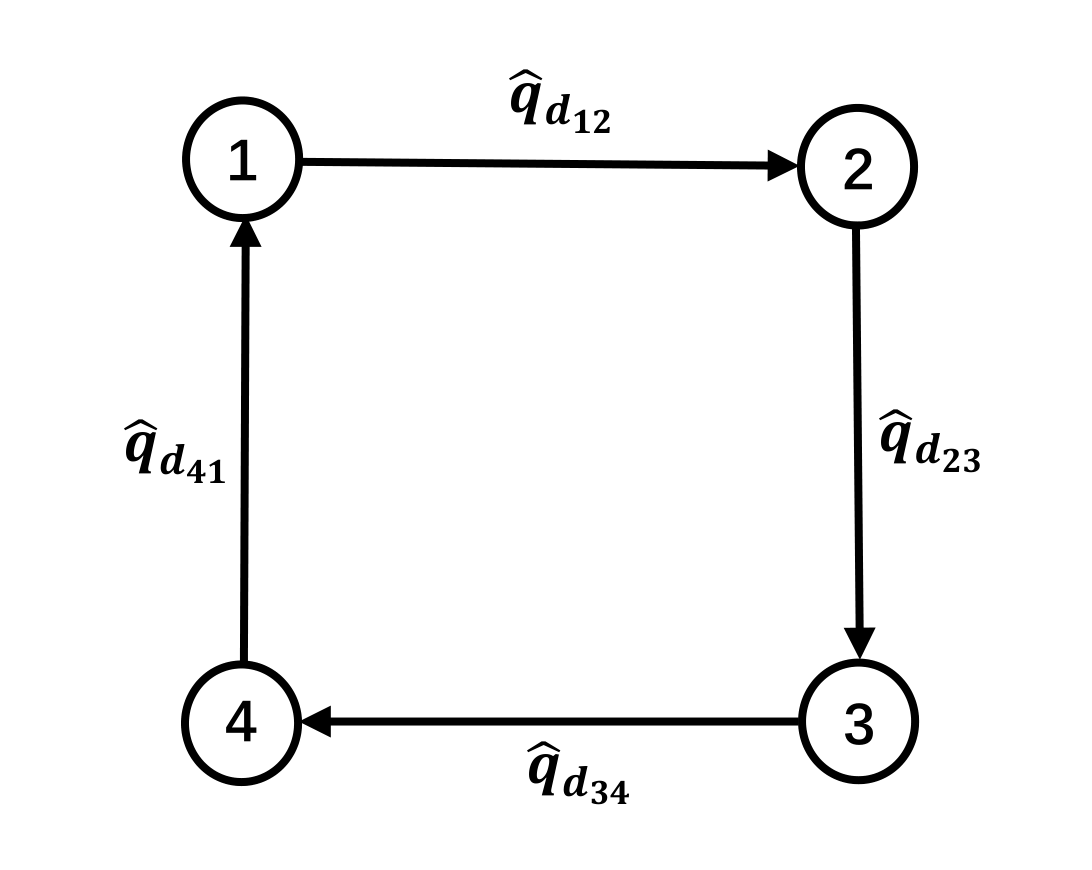}
        \caption{$|V|=4$}
        \label{fig1:sub2}
    \end{subfigure}
    \caption{Illustrations of two directed cycles.}    
    \label{fig1:main}
\end{figure}

In contrast to the Hermitian cases presented in \cite{cui2024power}, PM may not converge for the Laplacian matrices even if the cycle is balanced. 

\begin{example}\label{ex:L3c}
   Consider the Laplacian matrix $\hat{L}_3^{c}$ defined by \eqref{Lap:cycle}. Suppose  $\hat{q}_{d_{12}},\hat{q}_{d_{23}}$ are random unit dual quaternions, and set  $\hat{q}_{d_{31}}=\hat{q}_{d_{23}}^*\hat{q}_{d_{12}}^*$, with the following values:
   \begin{equation*}
   \begin{aligned}
   \hat{q}_{d_{12}} &= -0.68 + 0.60\ii + 0.41\jj - 0.12\kk + (-0.22 - 0.47\ii + 0.32\jj + 0.02\kk)\epsilon,\\
   \hat{q}_{d_{23}} &= 0.89 + 0.32\ii - 0.03\jj + 0.32\kk + (-0.25 + 0.39\ii - 0.30\jj + 0.29\kk)\epsilon,\\
   \hat{q}_{d_{31}} &= -0.75 - 0.44\ii - 0.16\jj + 0.47\kk + (0.05 + 0.72\ii - 0.32\jj + 0.65\kk)\epsilon.
   \end{aligned}
   \end{equation*}
   The initial vector is chosen as:
  \begin{equation*}
  \hat{\vv}^{(0)} = \begin{pmatrix}
  0.58 + 0.38\ii - 2.28\jj - 0.70\kk\\
  -0.27 + 1.71\ii + 1.29\jj + 0.06\kk\\
  -0.50 - 0.23\ii + 1.26\jj - 0.61\kk
  \end{pmatrix} +
  \begin{pmatrix}
  -0.63 + 0.06\ii - 0.09\jj - 0.42\kk\\
  -0.53 + 1.26\ii + 1.25\jj + 0.76\kk\\
  0.82 - 3.16\ii + 1.30\jj - 0.00\kk
  \end{pmatrix}\epsilon.
  \end{equation*}
   We compute the strict dominant eigenvalue of $\hat{L}^c_3$ using PM and DCAM-PM. 
   The results  presented in \textup{Fig.~\ref{fig2:main}~(a)} show that the residuals of both methods fail to converge to zero.. 
   \end{example}

   In fact, the eigenvalues of $\hat{L}_3^c$ in Example~\ref{ex:L3c} has closed form expressions.  According to \cite{qi2025unit}, the directed cycle graph is balanced. Consequently, the eigenvalues of the operator $\hat{L}_3$ coincide with those of the Laplacian matrix associated with its underlying undirected graph  $L_3=\begin{pmatrix}
        1&-1&0\\
        0&1&-1\\
        -1&0&1
    \end{pmatrix}
    $, whose eigenvalues of $L_3$ are $\frac{3}{2}+\frac{\sqrt{3}}{3}\ii,\frac{3}{2}-\frac{\sqrt{3}}{3}\ii$ and $0$.  Therefore, the residuals of PM and DCAM-PM do not converge to zero. 

    \begin{figure}[htbp]
    \centering
    \begin{subfigure}{0.48\textwidth}
        \centering
        \includegraphics[width=\textwidth]{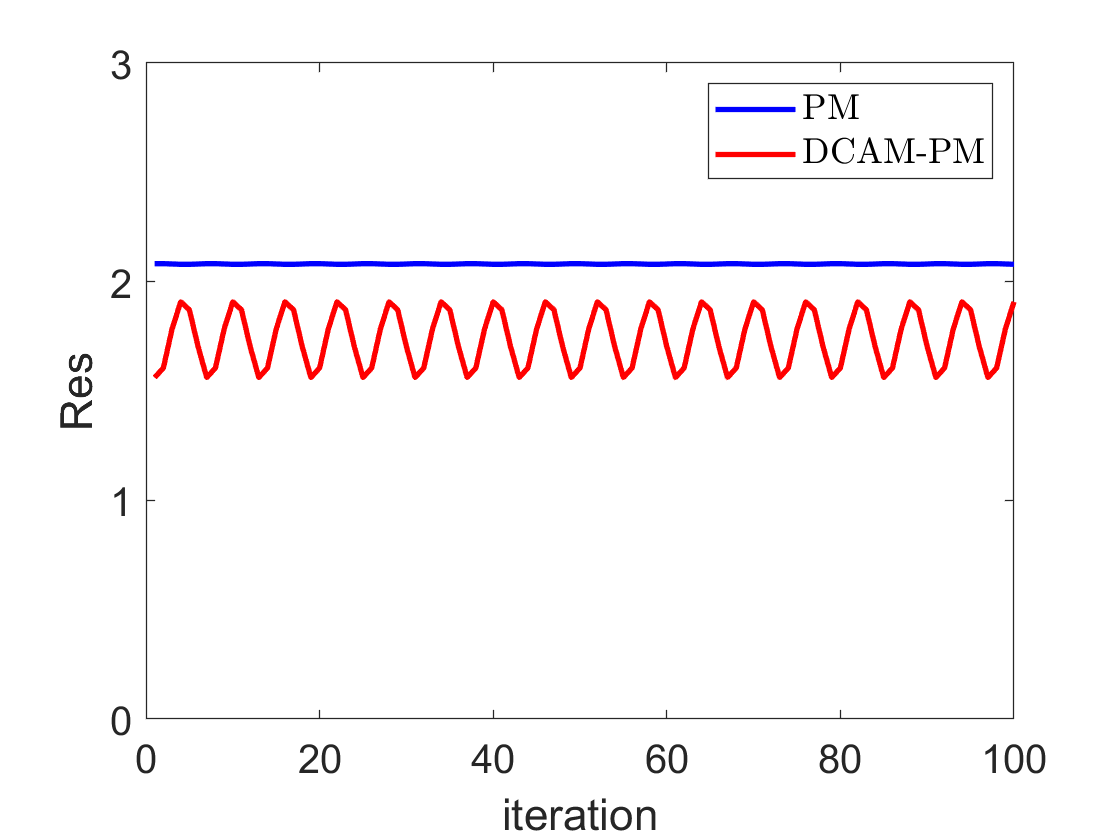}
        \caption{$|V|=3$}
        \label{fig2:sub1}
    \end{subfigure}
    \begin{subfigure}{0.48\textwidth}
        \centering
        \includegraphics[width=\textwidth]{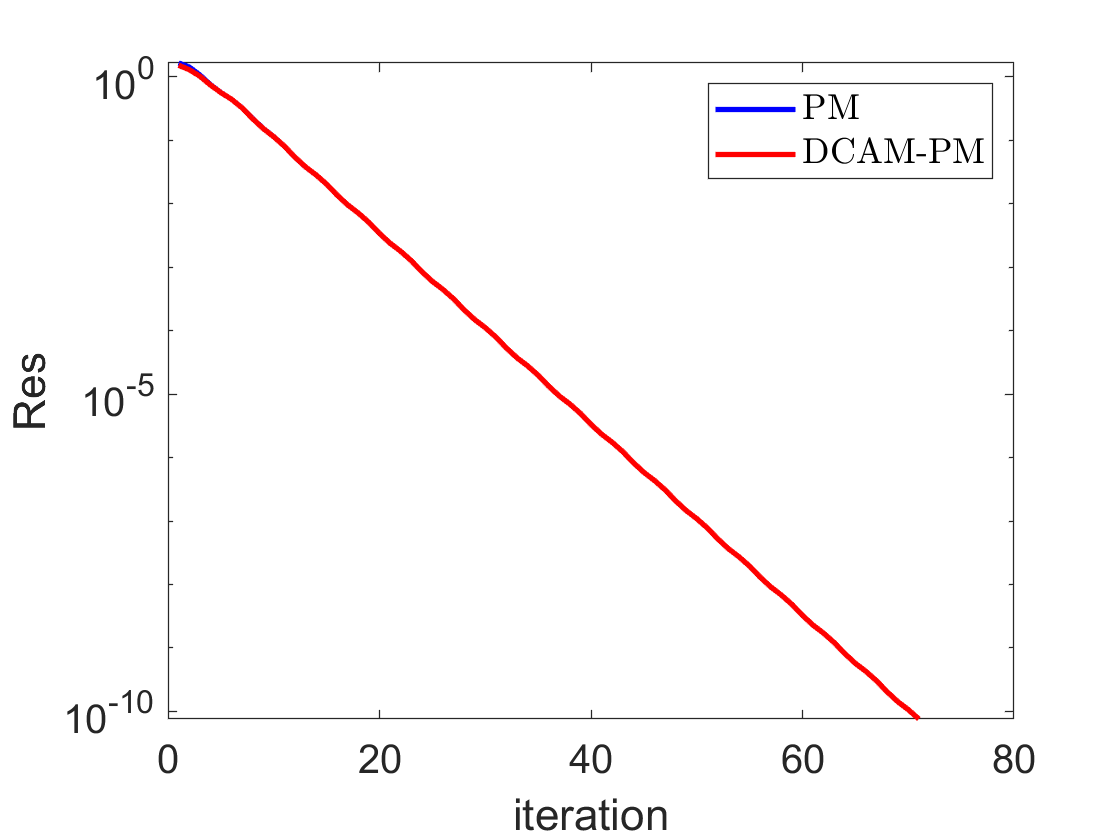}
        \caption{$|V|=4$}
        \label{fig2:sub2}
    \end{subfigure}
    \caption{Numerical results of PM and DCAM-PM for Laplacian matrices of  balanced directed cycles defined by \eqref{Lap:cycle}.}  
    \label{fig2:main}
    \end{figure}

\begin{example}
    Consider the Laplacian matrix $\hat{L}_4^{c}$ defined by \eqref{Lap:cycle}. Suppose   $\hat{q}_{d_{12}},\hat{q}_{d_{23}},\hat{q}_{d_{34}}$ are random unit dual quaternions, and set  $\hat{q}_{d_{41}}=\hat{q}_{d_{34}}^*\hat{q}_{d_{23}}^*\hat{q}_{d_{12}}^*$, with the following values:
   \begin{equation*}
\begin{aligned}
\hat{q}_{d_{12}} &= -0.34 - 0.15\ii + 0.81\jj - 0.46\kk + (-0.17 + 0.42\ii + 0.03\jj + 0.05\kk)\epsilon, \\
\hat{q}_{d_{23}} &= 0.06 + 0.82\ii - 0.50\jj - 0.27\kk + (-0.57 - 0.18\ii - 0.54\jj + 0.32\kk)\epsilon, \\
\hat{q}_{d_{34}} &= 0.14 - 0.40\ii + 0.56\jj + 0.71\kk + (-0.40 - 0.61\ii + 1.16\jj - 1.17\kk)\epsilon, \\
\hat{q}_{d_{41}} &= 0.25 + 0.11\ii - 0.91\jj + 0.30\kk + (-1.08 - 0.72\ii - 0.12\jj + 0.82\kk)\epsilon.
\end{aligned}
\end{equation*}
   The initial vector is chosen as:
   \begin{equation*}
\hat{\vv}^{(0)}=\begin{pmatrix}
2.21+1.13\ii-0.03\jj-0.02\kk\ \\
-1.34+1.40\ii-0.97\jj-1.52\kk\ \\
-1.49+0.75\ii+0.03\jj-0.07\kk\ \\
-0.03+0.06\ii+0.20\jj+0.82\kk \\
\end{pmatrix}+
\begin{pmatrix}
-0.64-0.71\ii-1.17\jj+0.54\kk\ \\
-1.41+0.39\ii-0.88\jj+1.90\kk\ \\
-1.28+0.38\ii+2.72\jj-0.18\kk\ \\
-0.42+1.52\ii-0.99\jj-0.16\kk \\
\end{pmatrix}\epsilon.
\end{equation*}
   We compute the strict dominant eigenvalue of $\hat{L}^c_4$ using the PM and DCAM-PM. The result is presented in \textup{Fig.~\ref{fig2:main}~(b)}. 
\end{example}
Since the directed cycle is balanced, $\hat{L}_4^{c}$ is similar to the Laplacian matrix of the underlying graph, which is $\begin{pmatrix}
        1&-1&0&0\\
        0&1&-1&0\\
        0&0&1&-1\\
        -1&0&0&1
    \end{pmatrix}
    $. Its standard eigenvalues are $2,1+\ii,1+\ii,0$. Thus $\hat{L}_4^c$ satisfies Assumption \ref{assump2}~(i). Both theoretical analysis and Fig.~\ref{fig2:main}~(b) demonstrate the convergence of PM and DCAM-PM with a rate of $\tilde{O}_D\left(\left(\frac{\sqrt{2}}{2}\right)^k\right)$. 

In general, for a unit dual quaternion balanced directed cycle with $n$ vertex, its Laplacian matrix $\hat{L}_n^{c}\in\mathbb{DQ}^{n\times n}$ is similar to the Laplacian matrix of its underlying graph:
\begin{equation}\label{eq:Lap}
    L_n^{c}=\begin{pmatrix}
        1&-1&0&\cdots&0&0\\
        0&1&-1&\cdots&0&0\\
        0&0&1&\cdots&0&0\\
        \vdots&\vdots&\vdots&\ddots&\vdots&\vdots\\
        0&0&0&\cdots&1&-1\\
        -1&0&0&\cdots&0&1
    \end{pmatrix}.
\end{equation}
When $n$ is odd, its eigenvalues are given by $1+e^{\frac{(2k-1)\pi\ii}{n}}$ for $k=0,1,\dots,n-1$. Consequently, $\hat{L}_n^{c}$ does not satisfy Assumption~\ref{assump2}. Conversely, when $n$ is even, the eigenvalues are $1+e^{\frac{2k\pi\ii}{n}}$ for $k=0,1,\dots,n-1$, and in this case,   $\hat{L}_n^{c}$ satisfies Assumption~\ref{assump2}~(i). Therefore, PM and DCAM-PM fail to compute the eigenvalues of  $\hat{L}_n^{c}$ When $n$ is odd, but succeed when $n$ is even.

A directed wheel graph comprises a central vertex, a directed cycle of surrounding vertices, and directed edges connecting the center to each surrounding vertex. We consider the Laplacian matrices of such directed wheel graphs. \textup{Fig.}~\ref{fig3:main} shows the directed cycles with 4 vertices and 5 vertices. 
\begin{figure}[htbp]
    \centering
    \begin{subfigure}{0.3\textwidth}
        \centering
        \includegraphics[width=\textwidth]{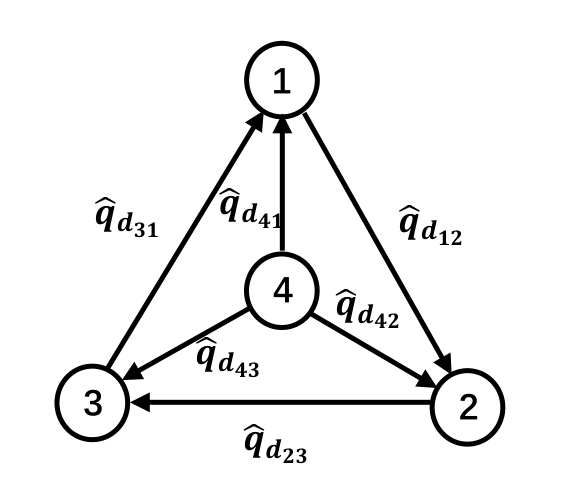}
        \caption{$|V|=4$}
        \label{fig3:sub1}
    \end{subfigure}
    \begin{subfigure}{0.31\textwidth}
        \centering
        \includegraphics[width=\textwidth]{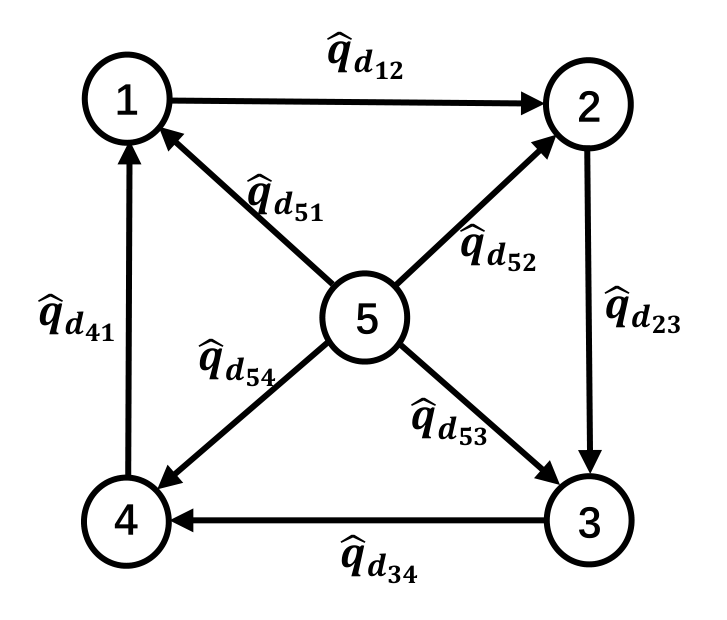}
        \caption{$|V|=5$}
        \label{fig3:sub2}
    \end{subfigure}
   \caption{Illustrations of two directed wheel graphs.}   \label{fig3:main}
\end{figure}
\par
For a unit dual quaternion balanced directed wheel graph with $n(n\geq 4)$ vertex, its Laplacian matrix $\hat{L}_n^{w}\in\mathbb{DQ}^{n\times n}$ is similar to the Laplacian matrix of its underlying graph:
\begin{equation}\label{L:wheel}
    L_n^w=\begin{pmatrix}
        L_{n-1}^c&\mathbf{0}\\
        \mathbf{s}&n-1
    \end{pmatrix},
\end{equation}
where $L_{n-1}^c$ is defined as in equation \eqref{eq:Lap}, $\mathbf{s}=(1\ 1\ \cdots\ 1)$. $\hat{L}_n^{w}$ satisfies Assumption~\ref{assump2}~(i) and $n-1$ is the strict dominant eigenvalue of $\hat{L}_n^{w}$ , regardless of whether $n$ is odd or even.
\begin{example}
    Random balanced unit dual quaternion directed wheel graphs with $4$ and $5$ vertices are generated, with Laplacian matrices $\hat{L}_4^{w}$ and $\hat{L}_5^{w}$, respectively. The eigenvalues of $\hat{L}_4^{w}$ are $3,\frac{3}{2}+\frac{\sqrt{3}}{3}\ii,\frac{3}{2}+\frac{\sqrt{3}}{3}\ii$ and $0$. The eigenvalues of $\hat{L}_4^{w}$ are $4,2,1+\ii,1+\ii$ and $0$. We apply PM and DCAM-PM to compute the strict dominant eigenvalues of $\hat{L}_4^{w}$ and $\hat{L}_5^{w}$. The results are presented in \textup{Fig.~\ref{fig4:main}}. As shown in \textup{Fig.~\ref{fig4:main}}, the residuals of both PM and DCAM-PM converge to zero at a rate of $\tilde{O}_D\left(\left(\frac{\sqrt{3}}{3}\right)^k\right)$ for $|V|=4$ and $\tilde{O}_D\left(\left(\frac{1}{2}\right)^k\right)$ for $|V|=5$.
\end{example}

    \begin{figure}[h]
    \centering
    \begin{subfigure}{0.4\textwidth}
        \centering
        \includegraphics[width=\textwidth]{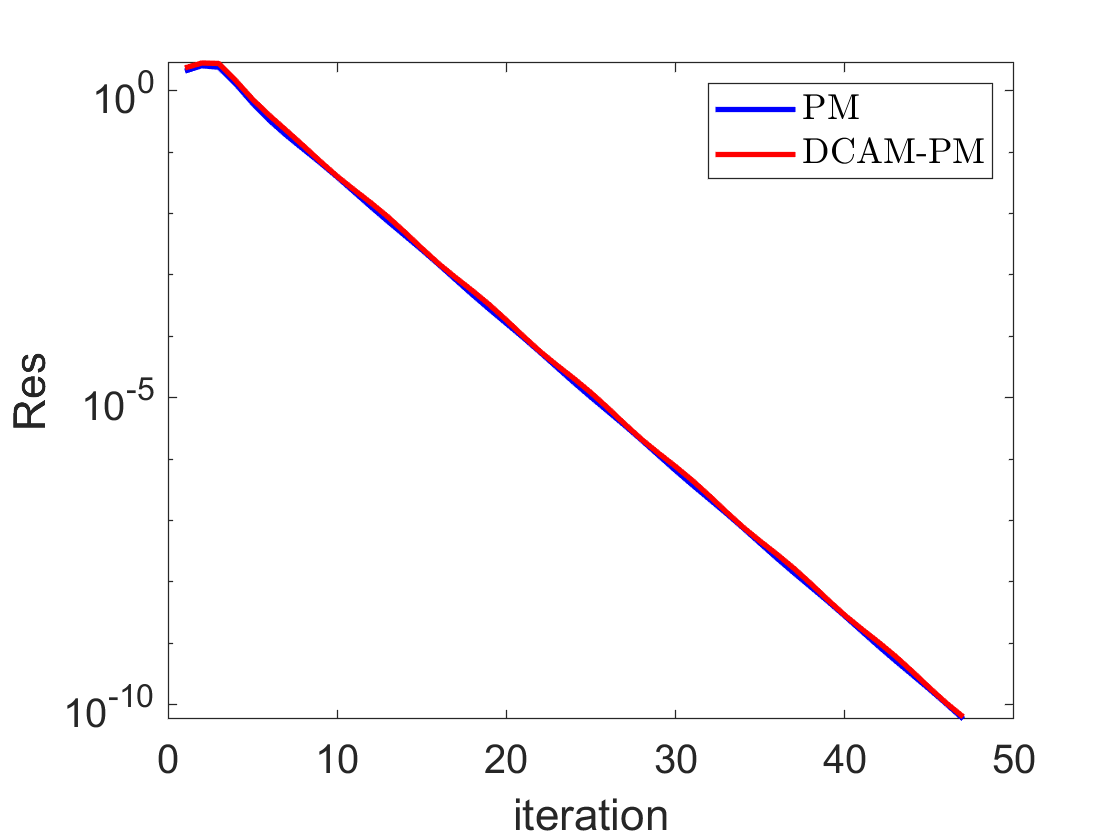}
        \caption{$|V|=4$}
        \label{fig4:sub1}
    \end{subfigure}
    \begin{subfigure}{0.4\textwidth}
        \centering
        \includegraphics[width=\textwidth]{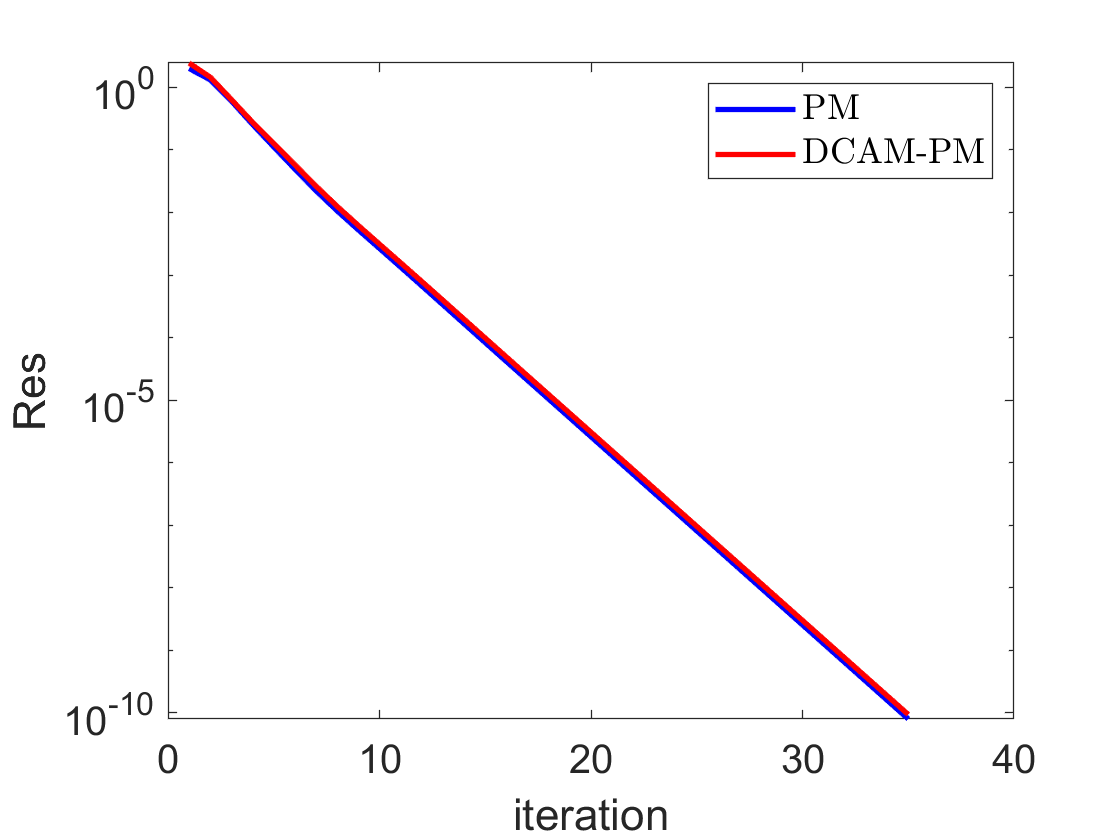}
        \caption{$|V|=5$}
        \label{fig4:sub2}
    \end{subfigure}
    \caption{Numerical results for directed wheel graphs, using the Laplacian matrix  in \eqref{L:wheel}.}    \label{fig4:main}
    \end{figure}

\subsection{Comparison of PM and DCAM-PM}
To elucidate the performance differences between PM and DCAM-PM, we constructed a series of structured matrices with tailored properties. These test cases are designed to isolate key scenarios which demonstrate that, although the algorithms exhibit similar convergence behavior for Hermitian DQMs, DCAM-PM imposes notably more stringent convergence conditions for non-Hermitian cases compared to PM. 

\begin{example}
    In this experiment, we randomly generate a diagonalizable DQM $\hat{A}\in\mathbb{DQ}^{10\times 10}$ with eigenvalues $2+\ii+\epsilon,1+\ii+\epsilon,\dots,1+\ii+\epsilon$. $\hat{A}$ satisfies Assumption~\ref{assump2} but does not satisfy Assumption~\ref{assump2}~(i). We select an initial vector randomly and compute the strict dominant eigenvalue of $\hat{A}$ via PM and DCAM-PM. The result is presented in Fig.~\ref{fig:compare}.
\end{example}
\begin{figure}[htbp]
    \centering
    \includegraphics[width=0.55\textwidth]{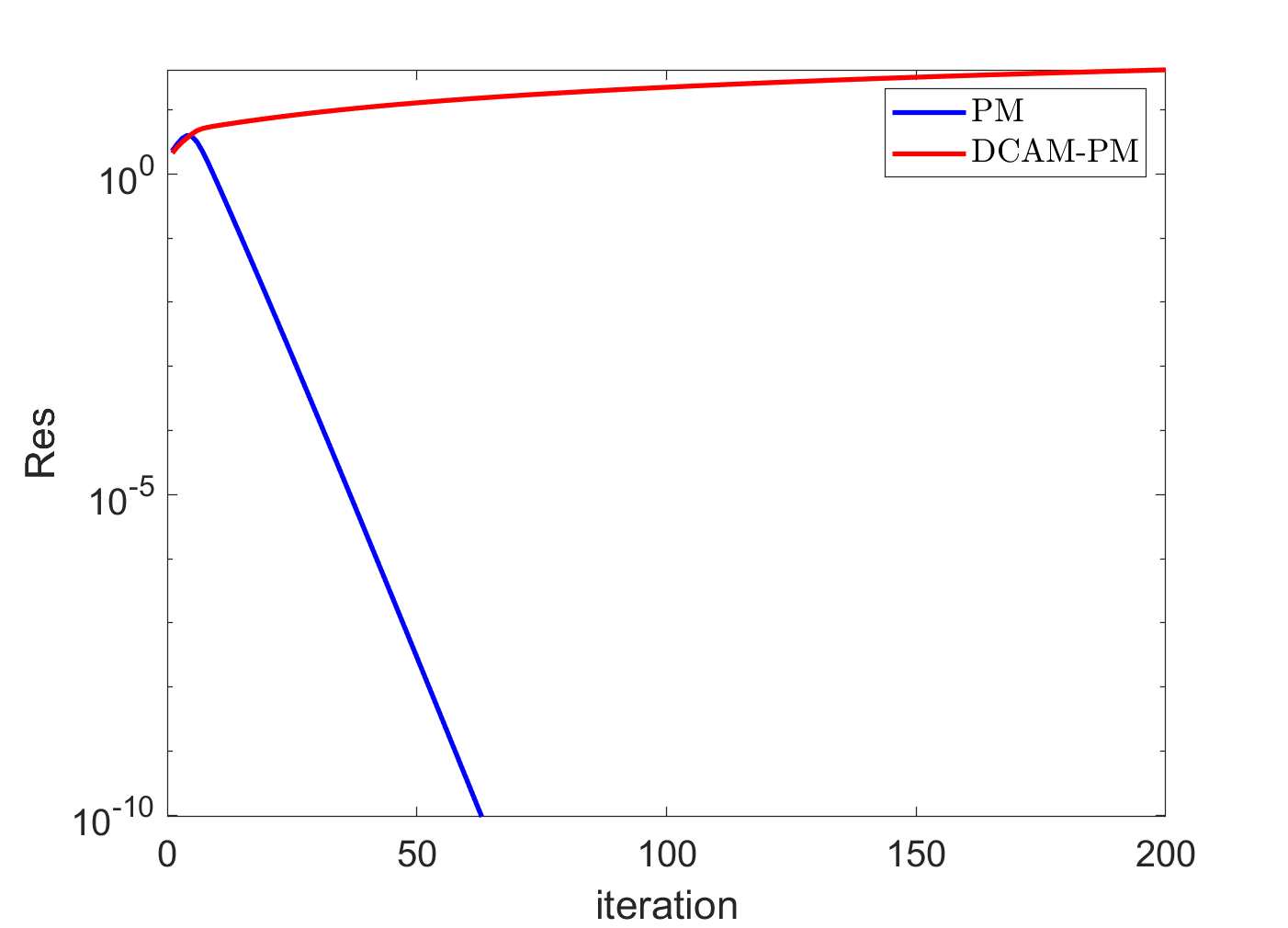}
   \caption{Numerical results for a DQM satisfying Assumption~\ref{assump2} but not Assumption~\ref{assump2}~(i).}   \label{fig:compare}
\end{figure}
From Fig.~\ref{fig:compare}, the residual of PM converges to zero with a convergence rate of $\tilde{O}_D\left(\left(\frac{1}{2}\right)^k\right)$, whereas the residual of DCAM-PM fails to converge to zero. This is consistent with Theorems~\ref{Thm:divergence_quater} and  \ref{Thm:divergence_adjoint}.

Next, we continue to investigate the convergence properties of the power method using structured matrices with non-diagonalizable standard parts. 
While the linear convergence rate is governed by the ratio $\frac{|\lambda_{2s}|}{|\lambda_{1s}|}$, the constant factor in the big-O notation is ultimately dictated by the size of the largest Jordan block, as established in equation~\eqref{eq:proof2}. 
This distinction arises directly from the non-Hermitian nature of the problem, as Hermitian DQMs are guaranteed to be diagonalizable.

\begin{example}
    Let $\hat{B}=B_s+B_d\epsilon\in\mathbb{DQ}^{10\times 10}$ be a non-Hermitian DQM, where
	\begin{equation*}
		B_s=\begin{pmatrix}
			1.1+1.1\ii&&\\
			&J_{n_{21}}(1+\ii)&\\
			&&I_{9-{n_{21}}}(1+\ii)
		\end{pmatrix} \text{  and  } 
        B_d=I_{10}.
	\end{equation*} 
     Here, $J_{n_{21}}(1+\ii)$ is a Jordan block of order ${n_{21}}$ and  $I_{9-{n_{21}}}$ is an identity matrix. We begin by constructing a matrix $\hat{A}$ that satisfies Assumption~\ref{assump2} via the similarity transformation $\hat{A}=\hat{P}^{-1}\hat{B}\hat{P}$, where $\hat{P}\in\mathbb{DQ}^{10\times 10}$ is a random invertible DQM. Then, for  ${n_{21}}=1,3,6,9$, we compute the strictly dominant eigenvalue of $\hat{A}$ using PM. The convergence history, measured by the residual $\textup{Res}$, is plotted against the iteration steps in Fig.~\ref{fig:Jordan:main}. 
    \begin{figure}[htbp]
    \centering
    \begin{subfigure}{0.24\textwidth}
        \centering
        \includegraphics[width=\textwidth]{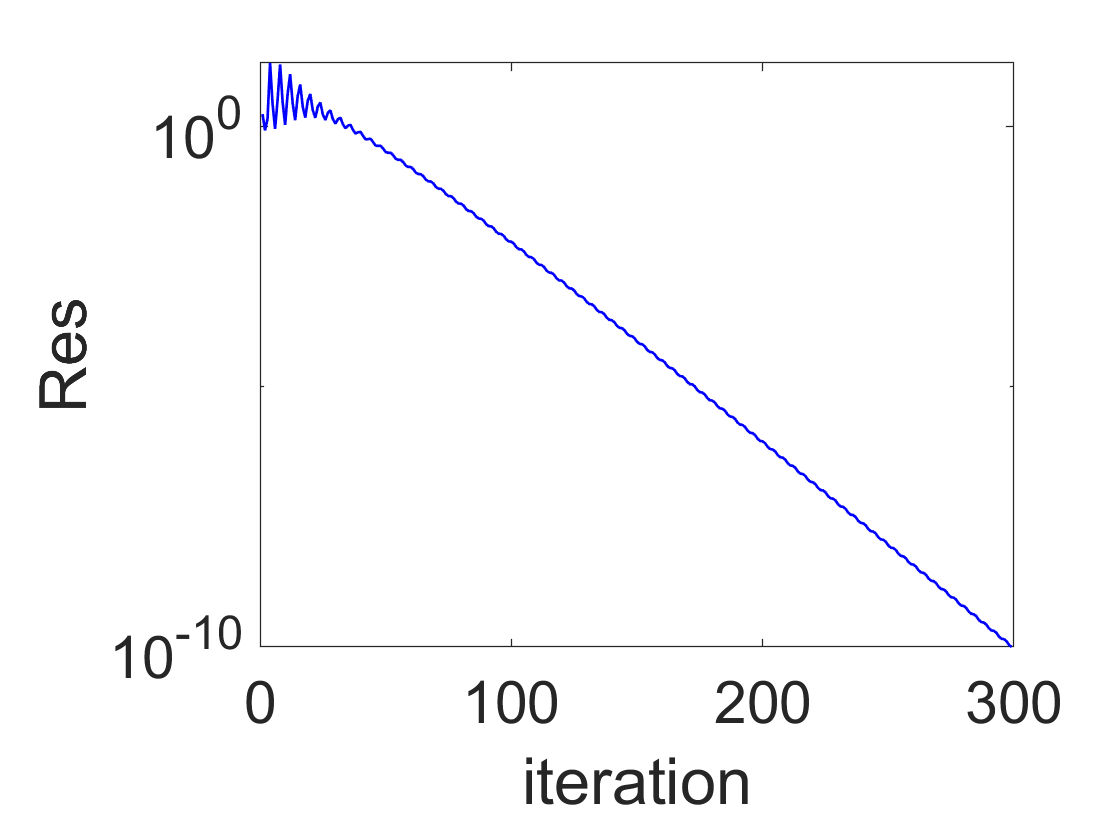}
        \caption{$n_{21}=1$}
        \label{fig:Jordan:sub1}
    \end{subfigure}
    \begin{subfigure}{0.24\textwidth}
        \centering
        \includegraphics[width=\textwidth]{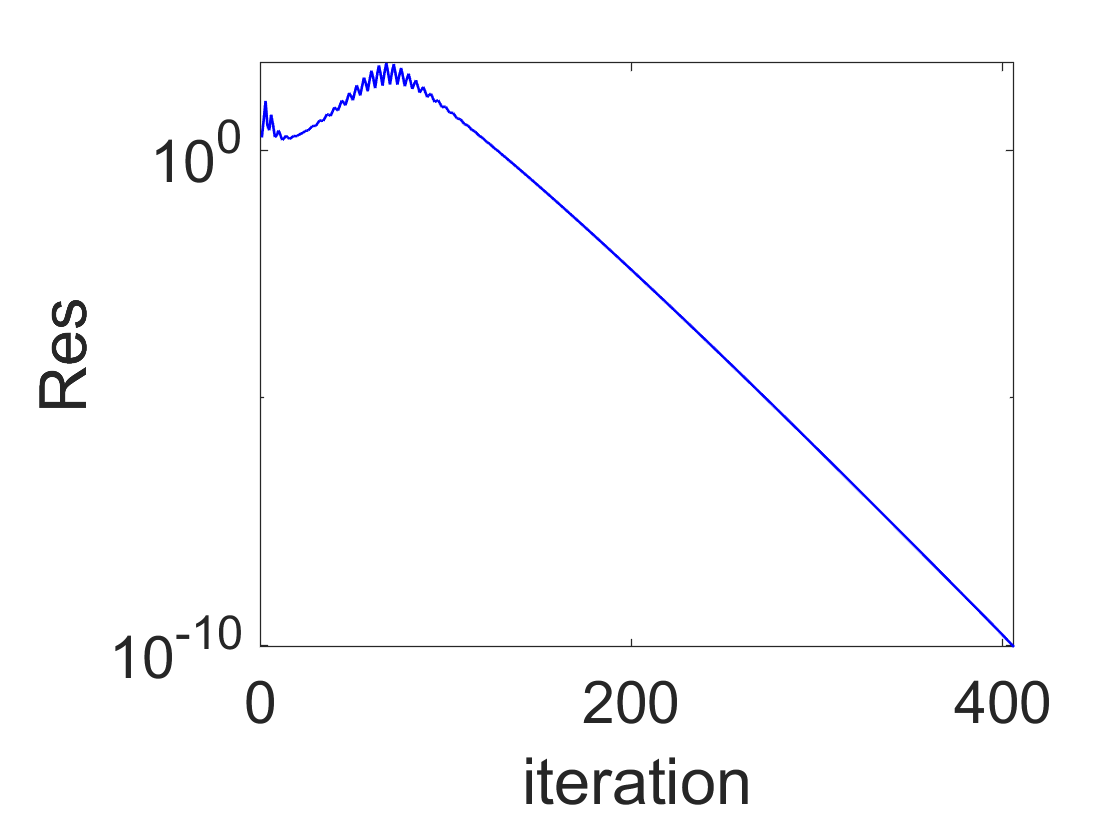}
        \caption{$n_{21}=3$}
        \label{fig:Jordan:sub2}
    \end{subfigure}
    \begin{subfigure}{0.24\textwidth}
        \centering
        \includegraphics[width=\textwidth]{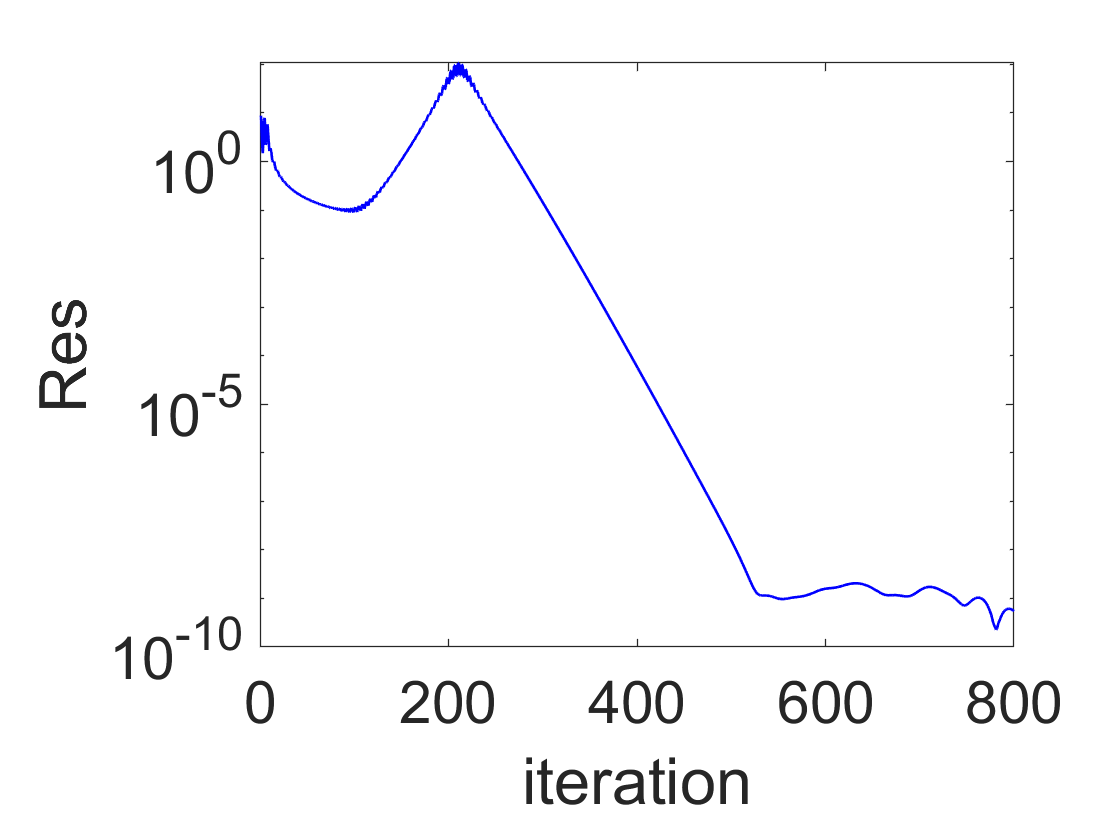}
        \caption{$n_{21}=6$}
        \label{fig:Jordan:sub3}
    \end{subfigure}
    \begin{subfigure}{0.24\textwidth}
        \centering
        \includegraphics[width=\textwidth]{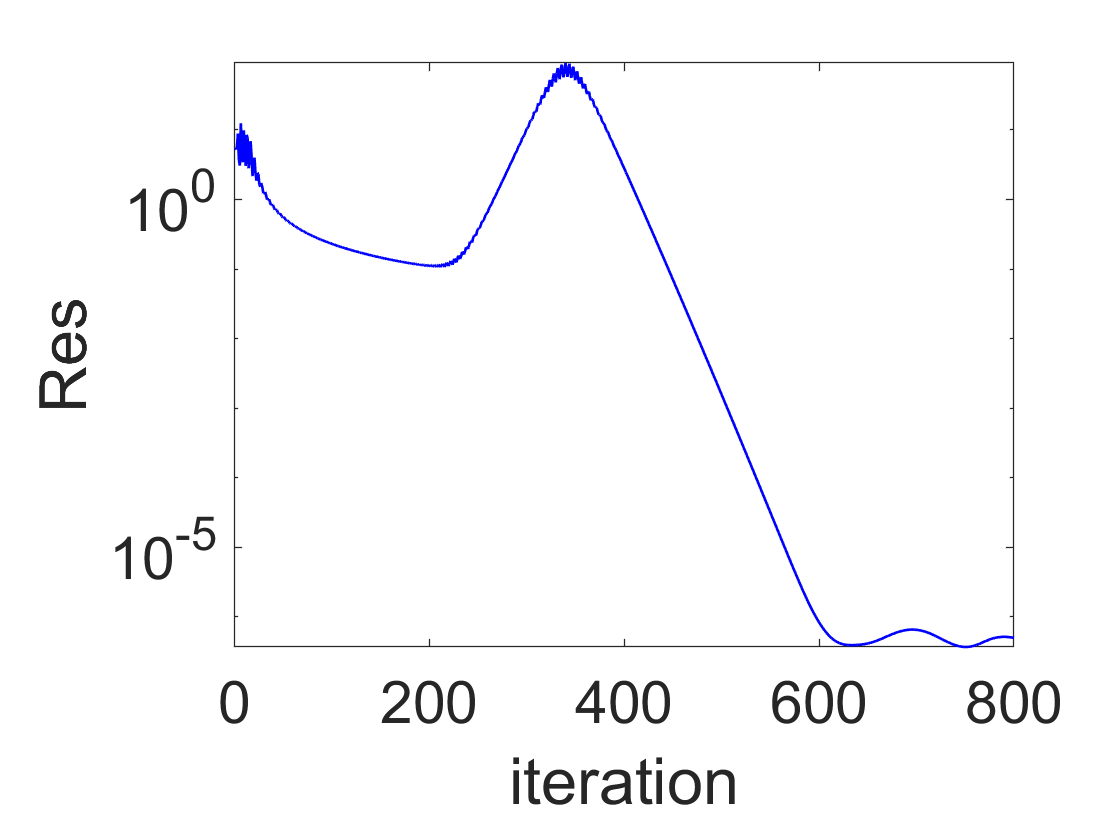}
        \caption{$n_{21}=9$}
        \label{fig:Jordan:sub4}
    \end{subfigure}
    \caption{Results for computing the strict dominant eigenvalue of the DQM whose standard part is not diagonalizable by the PM.}    \label{fig:Jordan:main}
    \end{figure}
\end{example}
From these figures, it can be seen that the PM always converges, but the fluctuations of PM are becoming increasingly large before converging as the order of the Jordan block increases. In addition, the accuracy achievable by the power method deteriorates. The higher the order of the Jordan block, the worse the precision with which it converges.

Finally, we assess the computational efficiency for test cases satisfying Assumption~\ref{assump2}~(i), where both PM and DCAM-PM exhibit linear convergence.

\begin{example}
    In this experiment, we randomly generate a dual quaternion diagonalizable matrix $\hat{A}\in\mathbb{DQ}^{n\times n}$ with eigenvalues $1.5+\epsilon,1+\epsilon,\dots,1+\epsilon$. For $n=10,50,100,200,500$, we compute the strict dominant eigenvalue of $\hat{A}$ using the PM and DCAM-PM. All experiments are repeated ten times with different choices of $\hat{A}$. We also denote ‘Iter’ as the average number of iterations and ‘Time (s)’ as the average
    CPU time in seconds for the PM and DCAM-PM. We show the performance of the PM and DCAM-PM in Table~\ref{tab:performance}. 
 Although both algorithms share the same computational complexity, they demonstrate distinct computational speeds in the numerical experiments.
\end{example}

\begin{table}[htbp]
  \centering
  \setlength{\tabcolsep}{12pt}
  \caption{Performance of the PM and DCAM-PM for DQMs satisfying Assumption~\ref{assump2}~(i).}  
  \begin{tabular}{c c c c c c c}
    \toprule
     & \multicolumn{3}{c}{PM} & \multicolumn{3}{c}{DCAM-PM} \\
    $n$ & Res & Iter & Time (s) & Res & Iter & Time (s) \\
    \midrule
    10   & 8.45e$-$11 & 62.4 & 1.63e$-$2 & 7.77e$-$11 & 66 & 3.40e$-$3 \\
    20   & 8.07e$-$11 & 62.5 & 1.64e$-$2 & 8.50e$-$11 & 67  & 3.70e$-$3 \\
    50  & 8.51e$-$11 & 62.7 & 1.77e$-$2 & 9.09e$-$11 & 68.2  & 5.30e$-$3 \\
    100  & 8.26e$-$11 & 62.8  & 1.97e$-$2 & 8.25e$-$11 & 69.7  & 1.40e$-$2\\
    200  & 8.18e$-$11 & 63  & 3.23e$-$2 & 7.79e$-$11 & 70  & 1.34e$-$1\\
    500  & 8.39e$-$11 & 62.1  & 9.07e$-$2 & 8.32e$-$11 & 71.1  & 7.43e$-$1\\
    \bottomrule
  \end{tabular}
  \label{tab:performance}  
\end{table}

\section{Conclusion}
This paper investigated the computation of the strict dominant eigenvalue for non-Hermitian DQMs via the power method (PM, Algorithm~\ref{alg_PM}) and DCAM  based power method (DCAM-PM, Algorithm~\ref{alg_PM_adj}). By introducing a novel Jordan-type decomposition, we establish the linear convergence of the power method under Assumption~\ref{assump2} and prove that this assumption is both necessary and sufficient.  While the convergence rate of DCAM-PM is comparable to PM, its convergence conditions are more stringent. 
Our experimental results confirmed the above theoretical analysis and demonstrated the efficiency of PM and DCAM-PM. 
Our code is available at \url{https://github.com/BUAA-HaoYang/DQ-toolbox}.

Future work could focus on extending the proposed approach to other iterative methods such as inverse iteration and QR-type algorithms, with the additional aim of developing eigenvalue computation methods under conditions more relaxed than those specified in Assumption~\ref{assump2}.

\bibliographystyle{siam}
\bibliography{reference}
\end{document}